\documentclass[12pt]{amsart}
\usepackage[utf8]{inputenc}
\usepackage{amsfonts}
\usepackage{amssymb}
\usepackage{amsmath}
\usepackage[english]{babel}
\usepackage{amsthm}
\theoremstyle{plain}
\newtheorem{theorem}{Theorem}[section]
\newtheorem{corollary}[theorem]{Corollary}
\newtheorem{lemma}[theorem]{Lemma}

\theoremstyle{definition}
\newtheorem{definition}[theorem]{Definition}
\newtheorem{remark}[theorem]{Remark}

\usepackage{geometry}
\numberwithin{equation}{section} 

\usepackage{lipsum}

\newcommand\blfootnote[1]{%
  \begingroup
  \renewcommand\thefootnote{}\footnote{#1}%
  \addtocounter{footnote}{-1}%
  \endgroup
}

\usepackage[colorlinks, citecolor=blue,pagebackref,hypertexnames=false]{hyperref}

\newcounter{comcount}

\begin{document}

\title{Reverse Riesz Inequality on Manifolds with Ends} 
\date{}
\author{DANGYANG HE}  
\address{Department of Mathematics and Statistics, Macquarie University}
\email{dangyang.he@mq.edu.au}

\begin{abstract}
In our investigation, we focus on the reverse Riesz transform within the framework of manifolds with ends. Such manifolds can be described as the connected sum of finite number of Cartesian products $\mathbb{R}^{n_i} \times \mathcal{M}_i$, where $\mathcal{M}_i$ are compact manifolds. We rigorously establish the boundedness of this transform across all $L^p$ spaces for $1<p<\infty$. Notably, existing knowledge indicates that the Riesz transform in such a context demonstrates boundedness solely within a specific range of $L^p$ spaces, typically observed for $1<p<n_*$, where $n_*$ signifies the smallest dimension of the manifold's ends on a large scale.

This observation serves as a significant counterexample to the presumed equivalence between the Riesz and reverse Riesz transforms. Our study illuminates the nuanced behaviour of these transforms within the setting of manifolds with ends, providing valuable insights into their distinct properties. Although the lack of equivalence has been previously noted in relevant literature, our investigation contributes to a deeper understanding of the intricate interplay between the Riesz and reverse Riesz transforms.
\end{abstract}

\maketitle

\tableofcontents

\blfootnote{$\textit{2020 Mathematics Subject Classification.}$ 42B20, 47F05}

\blfootnote{$\textit{Keywords and Phrases.}$ Riesz transform, reverse Riesz inequality, manifolds with ends, Bessel functions.}
\section{Introduction}

The Riesz transform stands arguably as a primary research area in both harmonic analysis and partial differential equations, captivating the minds of esteemed scholars for over 100 years. In the early 20th century, Riesz demonstrated that the Hilbert transform operator is bounded on the real line \cite{R}, employing a holomorphic iteration argument. However, extending this argument to multi-dimensional Euclidean spaces posed considerable challenges, arguably leading to the emergence of Calderón–Zygmund theory \cite{CZ} and developing the pivotal notion of Calderón–Zygmund decomposition. It seems that Riesz's manuscript and related research field had a huge impact on the content of the Hardy-Littlewood-Pólya book \cite{HLP}, which can be placed between Riesz's one dimensional results and Calderón–Zygmund development.  

The research area stemming from the notion of the Riesz transform was significantly expanded by a question posed by Strichartz in \cite{S}. His work included advocating for the examination of the corresponding gradient inequality within the framework of complete Riemannian manifolds. The subject of the Riesz transform in the setting of Riemannian manifolds is so broad that it is impossible to provide a comprehensive bibliography of it here. So we refer readers to \cite{AC,ACDH,B2,B,CCH,CD2,CD,HS,L} and references within for more discussions. This research field is also closely related to the study of Riesz transform in the setting of exterior domains with Dirichlet or Neumann boundary conditions, see for example \cite{HS2,JL,KVZ} (we also discuss this point here because often one can use the same techniques as in complete Riemannian manifolds).

To delve deeper into Strichartz's idea, we briefly describe  standard definitions of gradient, Laplace-Beltrami operator, and Riesz transform in the setting of Riemannian Manifolds. 

Let $M$ be a complete Riemannian manifold endowed with the Riemannian volume measure $\mu$. Denote by $\nabla$ the Riemannian gradient such that $\langle X, \nabla f \rangle =Xf$ for any vector field $X$ and regular function $f$. 
The the corresponding quadratic (energy) form is defined by (see \cite{D} for detailed description about quadratic form)
\begin{equation*}
    Q(f,g)=\int_M \nabla f \cdot \nabla g  d\mu.
\end{equation*}
Then one can define $\Delta$, the Laplace-Beltrami operator, through the following formula (formal integration by parts)
\begin{equation}\label{eq_Laplacian}
    \int_M g \Delta f  = Q(f,g)  \quad \forall f,g\in \mathcal{C}_c^\infty(M).
\end{equation}
The Riesz transform can be then defined by $\nabla \Delta^{-1/2}$. Let $\|\cdot\|_p$ denote the norm in the space $L^p(M,d\mu)$ for $1\le p\le \infty$ and let $|\cdot|$ represent the form length, we say the Riesz transform is bounded on $L^p$ if
\begin{equation*}
    \left\| |\nabla \Delta^{-1/2}f| \right\|_p \le C \|f\|_p \quad \forall f\in \mathcal{C}_c^\infty(M)
\end{equation*}
or equivalently
\begin{equation}\tag{$\textrm{R}_p$}\label{R_p}
    \left\| |\nabla f| \right\|_p \le C \|\Delta^{1/2}f\|_p \quad \forall f\in \mathcal{C}_c^\infty(M).
\end{equation}
Note by \eqref{eq_Laplacian}, $(\textrm{R}_2)$ holds automatically.

The boundedness of $\nabla \Delta^{-1/2}$ also leads to the concept of reverse Riesz transform, which is a natural development of the study of Riesz transform. For example, in a homogeneous Sobolev space, one can ask about the equivalence of the semi-norm: $\left\| |\nabla f| \right\|_p$ and $\|\Delta^{1/2}f\|_p$ or the inequality
\begin{equation}\tag{$\textrm{E}_p$}\label{E_p}
    C_1\|\Delta^{1/2}f\|_{p} \le  \left\||\nabla f| \right\|_{p}\le C_2\|\Delta^{1/2}f \|_p \quad \forall f\in \mathcal{C}_c^\infty(M).
\end{equation}
Particularly, the left-hand inequality which we refer to as the reverse Riesz inequality 
\begin{equation}\tag{$\textrm{RR}_{p}$}\label{eq_RRp}
    \|\Delta^{1/2}f\|_{p} \le C \left\||\nabla f| \right\|_{p} \quad \forall f\in \mathcal{C}_c^\infty(M).
\end{equation}
For literature regarding reverse inequality, we refer readers to \cite{AC,CD,JL,DR}. Essentially, Riesz and reverse Riesz transform are expected to be equivalent. However, this is known to be false in the general setting (see \cite[Proposition 2.1]{CD}). Our investigation provides a significant example to this property and contributes to a deeper understanding of the interplay between the Riesz and reverse Riesz transforms.

Let us first review some previous results regarding \eqref{R_p}. For the case $p<2$, Coulhon and Duong \cite{CD2} proved that if $M$ satisfies the doubling condition
\begin{equation}\tag{$\textrm{D}$}\label{Doubling}
    \mu(B(x,2r)) \le C \mu(B(x,r)) \quad \forall x\in M \quad \forall r>0,
\end{equation}
where $B(x,r)$ is the geodesic ball centered at $x$ with radius $r$ and the heat kernel meets an on-diagonal estimate, then \eqref{R_p} holds for $1<p\le 2$. In the context of $p>2$, one considers a manifold $M$ satisfies \eqref{Doubling} and $(P_2)$ where \eqref{P_q} refers to the Poincaré inequality: for every ball~$B$ and $f\in \mathcal{C}^\infty(B)$
\begin{equation}\tag{$\textrm{P}_q$}\label{P_q}
    \int_{B} |f-f_B|^q d\mu \le C r^q \int_{B} |\nabla f|^q d\mu, 
\end{equation}
where $f_B = \mu(B)^{-1}\int_B fd\mu$ and $1\le q<\infty$. Then, it was verified by Auscher, Coulhon, Duong and Hofmann \cite{ACDH}, that \eqref{R_p} holds as $2<p<r$ for some $r>2$ provided the gradient of heat hernel satisfies some regularity.

Regarding the reverse Riesz inequality, it is well-known that \eqref{R_p} implies $(\textrm{RR}_{p'})$ (where $p'$ is the conjugate exponent of $p$ i.e. $1/p+1/p'=1$) and this property holds on any complete Riemannian manifolds, see \cite{CD}. However, the converse is not true in general. This property was observed in \cite{B} and formally proved in \cite{CD}. To attack the reverse inequality directly, Auscher and Coulhon \cite[Theorem 0.7]{AC} established a Sobolev version Calderón–Zygmund decomposition, also see \cite{A} under the setting of $\mathbb{R}^n$. Their findings suggest that if $M$ satisfies \eqref{Doubling} and \eqref{P_q} for some $1\le q<2$, then \eqref{eq_RRp} holds for $q<p\le 2$. It is also worth mentioning that by \cite{KVZ}, we know that for $d\ge 3$, $\|\Delta_\Omega^{s/2}f\|_p \le C \|\Delta_{\mathbb{R}^d}^{s/2}f\|_p$ for all $f\in \mathcal{C}_c^\infty(\Omega)$ and $0\le s<\min(1+1/p, d/p)$, where $\Omega$ is the complement of a convex compact set $\Omega^c \subset \mathbb{R}^d$ and $\Delta_\Omega$ denotes the Laplacian restricted to Dirichlet boundary condition on $\partial \Omega$. Particularly, their result shows that \eqref{R_p} (the Riesz transform, $\nabla \Delta_\Omega^{-1/2}$) holds for $p\in (1,d)$. Then, it follows by duality and the fact
\begin{equation*}
    \|\Delta_\Omega^{1/2}f\|_p \le C \|\Delta_{\mathbb{R}^d}^{1/2}f\|_p \le C \||\nabla f|\|_p\quad 1< p \le d'
\end{equation*}
that \eqref{eq_RRp} holds on $\Omega$ for all $p\in (1,\infty)$.

\medskip

Before presenting our main result, we introduce the formal notion of the connected sum. Investigations on manifolds with ends also attracts significant interests from experts in harmonic analysis. For more detailed discussions on the definition of such manifolds and related studies, we refer readers to \cite{BS,BDLW,CCH,D2,DLS,GS,HNS,HS,H,L,N} and the references therein.

\begin{definition}\label{1.2d}
Let $\mathcal{V}_i$ where $i = 1,2,\dots,l$ be a family of complete connected non-compact Riemannian manifolds with the same dimension. Then we call the Riemannian manifold $\mathcal{V}$ is the connected sum of $\mathcal{V}_1, \dots, \mathcal{V}_l$ and write it as
\begin{equation*}
    \mathcal{V} = \mathcal{V}_1\#  \mathcal{V}_2\# \dots \# \mathcal{V}_l
\end{equation*}
if for some compact subset $K \subset \mathcal{V}$, its compliment $\mathcal{V} \setminus K$ is a disjoint union of connected open sets $E_i$ where $i = 1,2,\dots,l$ such that each $E_i$ is isometric to $\mathcal{V}_i \setminus K_i$ for some compact set $K_i \subset  \mathcal{V}_i$. We call the subsets $E_i$ the $i$th end of $\mathcal{V}$. 
\end{definition}

In this note, we focus on the manifold appearing in the way
\begin{equation}\label{eq_Manifold}
    \mathcal{M} = (\mathbb{R}^{n_1} \times \mathcal{M}_1)\# \dots \#(\mathbb{R}^{n_l} \times \mathcal{M}_l) \quad l\ge 2 \quad n_i\ge 3,
\end{equation}
where $\mathcal{M}_i$ are some compact manifolds such that $n_i+\textit{dim}(\mathcal{M}_i)=N$ for all $i$. We define each end $E_i=(\mathbb{R}^{n_i}\times \mathcal{M}_i) \setminus K_i$ for some compact set $K_i\subset \mathbb{R}^{n_i}\times \mathcal{M}_i$. In what follows we denote by $d(\cdot,\cdot)$ the Riemannian distance function and $\mu$, the measure on $\mathcal{M}$. Note that $n_i$ are not necessarily all the same and the case when $n_i\ne n_j$ for some $i,j$ may cause the manifold non-doubling. Following the standard approach in \cite{CCH, HS}, we study the Riesz and reverse Riesz transforms via resolvent-based formulas
\begin{equation}\label{eq_riesz}
    \nabla \Delta^{-1/2} = \frac{2}{\pi} \int_0^\infty \nabla (\Delta+k^2)^{-1} dk.
\end{equation}
and
\begin{equation}\label{eq_delta-1/2}
    \Delta^{1/2} = \frac{2}{\pi} \int_0^\infty \Delta (\Delta+k^2)^{-1} dk.
\end{equation}

Most of results concerning Riesz and reverse Riesz transforms, see \cite{AC, ACDH, CD}, are based on two critical geometric features: \eqref{Doubling} and \eqref{P_q}. However, on $\mathcal{M}$, both \eqref{Doubling} and $(\textrm{P}_2)$ fail. Our method is based on a parametrix construction of the resolvent at low energy as in \cite{CCH, HS}, which we may discuss in the next section.


We briefly introduce some results regarding such manifolds with ends. It was proved by Carron, Coulhon and Hassell \cite[Theorem 1.1]{CCH} that the Riesz transform on $\mathbb{R}^n \# \mathbb{R}^n$ $(n\ge 3)$ is bounded on $L^p$ if and only if $p\in (1,n)$, where the study of two copies of planes are included in \cite{CD2}. Generally, on $\mathcal{M}$ (\eqref{eq_Manifold}), Hassell and Sikora \cite[Theorem 1.2]{HS} verified that \eqref{R_p} holds if and only if $p\in (1,n_*)$, where $n_* = \min_j n_j$ is the smallest Eulidean dimension among all the ends and a weak type $(1,1)$ result has been given. Sequentially, Hassell, Nix and Sikora \cite{HNS}, showed that on a critical but technically different case, where $ \mathcal{M}=(\mathbb{R}^2\times \mathcal{M}_-)\#(\mathbb{R}^{n_+}\times \mathcal{M}_+)$ with $n_+\ge 3$, \eqref{R_p} holds exactly in the range $p\in (1,2]$. Last but not the least, as a supplement, the author of \cite[Theorem 1.3]{H} completes the picture by showing $\nabla \Delta^{-1/2}$ is bounded on the Lorentz space $L^{n_*,1}$ and is unbounded from $L^{n_*,p}\to L^{n_*,q}$ for any $p\in (1,\infty)$ and $q\in[p,\infty]$. 

The main aim of this note is to prove following results.
\begin{theorem}\label{thm_MainResult}
Let $\mathcal{M}$ be a manifold with ends defined in \eqref{eq_Manifold} with $l\ge 2$ and $n_*:= \min_i n_i\ge 3$. Then, \eqref{eq_RRp} holds for all $p$ i.e.
\begin{equation}\label{eq_MainResult}
    \|\Delta^{1/2}f\|_p \le C \left\||\nabla f|\right\|_p \quad \forall f\in \mathcal{C}_c^\infty(\mathcal{M}).
\end{equation}
for all $p\in (1,\infty)$.
\end{theorem}

And for the critical case considered in \cite{HNS}, where 
\begin{align}\label{eq_manifold_criticalcase}
 \mathcal{M}=(\mathbb{R}^2\times \mathcal{M}_-)\#(\mathbb{R}^{n_+}\times \mathcal{M}_+)
\end{align}
with $n_+\ge 3$, a similar result holds.

\begin{theorem}\label{thm_RR_critical}
Let $\mathcal{M}$ be a manifold with ends defined by \eqref{eq_manifold_criticalcase}. Then, \eqref{eq_RRp} holds for all $1<p<\infty$.
\end{theorem}


\section{Parametrix Construction for Resolvent at Low Energy}

In this section, we briefly introduce the parametrix construction for the resolvent at low energy and we refer readers to \cite[Section 3]{HS} and \cite[Section 3,4]{CCH} for detailed analysis. The parametrix construction is based on the following lemma from \cite[Lemma 2.7]{HS}. Let $z_i^0$ be some fixed point in $K_i$ to play the role of origin. Without loss of generality, we may assume that $\inf_{x\in E_i}d(z_i^0,x)\ge 1$ for all $1\le i\le l$. In what follows, we will keep most of notations used in \cite{HS} for convenience. 

\begin{lemma}\cite[Lemma 2.7]{HS}\label{leKey}
Let $\mathcal{M}$ be defined by \eqref{eq_Manifold} with $l\ge 2$ and $n_* = \min_i n_i \ge 3$. Let $v \in L^{\infty}(\mathcal{M})$ with compact support in $K$. Then there is a function $u: \mathcal{M} \times \mathbb{R}^+ \to  \mathbb{R}$ such that $(\Delta + k^2)u = v$ and on the $i$th end
\begin{gather*}\label{2.7e}
|u(z, k)| \le  C \|v\|_\infty \begin{cases}
    1, & z\in K,\\
    d(z_i^0 , z)^{2-n_i} e^{- ck d(z_i^0, z)}, & z \in  E_i.
\end{cases}
\\
|\nabla u(z, k)| \le  C \|v\|_\infty \begin{cases}
    1, & z\in K,\\
    d(z_i^0 , z)^{1-n_i} e^{- ck d(z_i^0, z)}, & z\in E_i,
\end{cases}
\end{gather*}
for all $0\le k\le 1$, $1\le i\le l$. 

In addition, 
\begin{gather*}
    \|u(\cdot,k)-u(\cdot,0)\|_{L^\infty(\mathcal{M})}\le C k \|v\|_\infty,\\
    \left\||\nabla [u(\cdot,k)-u(\cdot,0)]| \right\|_{L^\infty(\mathcal{M})}\le C k \|v\|_\infty 
\end{gather*}
for all $0\le k\le 1$.

\end{lemma}

Now, for each $1\le i\le l$, we pick a cut-off function $\phi_i\in \mathcal{C}^\infty(\mathcal{M})$ which supported in the $i$th ends, $E_i$, such that $\phi_i=1$ outside a compact set containing the "unit ball" $K_i$. Define $v_i=-\Delta \phi_i$ which is apparently supported in a compact set near $K_i$. Then we set
\begin{equation}\label{eq_u}
    u_i(z,k):= (\Delta+k^2)^{-1}v_i \quad \forall 0\le k\le 1,
\end{equation}
where the existence of $u_i$ is guaranteed by Lemma \ref{leKey}.

The main idea is to decompose the resolvent into four pieces 
\begin{equation*}
    (\Delta+k^2)^{-1} = G_1(k) + G_2(k) + G_3(k) +G_4(k) \quad k\le 1.
\end{equation*}
Particularly, we write down explicit formulas for $G_1$ and $G_3$ for later use.
\begin{gather}\label{G1}
    G_1(k)(z,z') = \sum_{i=1}^l (\Delta_i +k^2)^{-1}(z,z') \phi_i(z) \phi_i(z'),\\ \label{G3}
    G_3(k)(z,z') = \sum_{i=1}^l (\Delta_i +k^2)^{-1}(z_i^0,z') u_i(z,k)\phi_i(z'),
\end{gather}
where $\Delta_i = \Delta_{\mathbb{R}^{n_i}\times \mathcal{M}_i}$ is the Laplace-Beltrami operator on $\mathbb{R}^{n_i}\times \mathcal{M}_i$. The constructions of the parametrix in \cite{CCH} and \cite{HS} are close but different and we follow the version in \cite{HS}. One of the main differences between those two constructions is how to handle the "error" term $\Tilde{E}(k)$ which satisfies equation
\begin{equation*}
    (\Delta+k^2)[G_1(k)+G_2(k)+G_3(k)] = I + \Tilde{E}(k).
\end{equation*}
In \cite{CCH}, Carron, Coulhon and Hassell added a finite rank term, which they also denote by $G_4$ ($G_4$ is not depending on $k$), to make operator $I+E(k)$ invertible for all sufficiently small $k$, where $E(k)$ is the new "error" term such that $I+E(k) = (\Delta+k^2)[G_1(k)+G_2(k)+G_3(k)+G_4]$. In \cite{HS}, Hassell and Sikora defined $G_4(k)$ via Lemma~\ref{leKey} by writing $G_4(k)(z,z'):= -(\Delta+k^2)^{-1}(\Tilde{E}(k)(\cdot,z'))(z)$. The method used in \cite{HS} can be applied to \cite{CCH} when $n\ge 3$.

We summarize some results and estimates from \cite{HS} in the following.

$(\romannumeral1)$ $G_2(k)(z,z')$ is compactly supported in a small neighbourhood of the diagonal $K \times K$. It is a family of pseudodifferential operator with order $-2$ depending on $k$ smoothly. Therefore, $\int_0^1 \nabla G_2(k)dk$ acts as a pseudodifferential operator with order $-1$ which is bounded on all $L^p$ i.e.
\begin{equation}\label{eq_G2_bdd}
    \left\|\left|\int_0^1 \nabla G_2(k)fdk\right|\right\|_p \le C\|f\|_p \quad \forall 1<p<\infty \quad \forall f\in \mathcal{C}_c^\infty(\mathcal{M}).
\end{equation}

$(\romannumeral2)$ $G_4(k)$ is the "correction" term whose kernel has nice decay property and satisfies the equation
\begin{equation*}
    (\Delta+k^2)[G_1(k)+G_2(k)+G_3(k)] = \textrm{Id} - (\Delta+k^2)G_4(k).
\end{equation*}
After careful calculation, it turns out that its contribution to the Riesz transform is also bounded on all $L^p$ i.e.
\begin{equation}\label{eq_G4_bdd}
    \left\|\left|\int_0^1 \nabla G_4(k)f dk\right| \right\|_p \le C \|f\|_p \quad \forall 1<p<\infty \quad \forall f\in \mathcal{C}_c^\infty(\mathcal{M}).
\end{equation}

$(\romannumeral3)$ Note that for each $i$, $v_i=-\Delta \phi_i$ so the function $\Phi_i(z):= \phi_i(z) + u_i(z,0)$ is a harmonic function which tends to 1 at the infinity of $E_i$ and to 0 at the infinity of $E_j$ for all $j\ne i$.

$(\romannumeral4)$ For each $i$, we have following resolvent estimates \cite[(17)-(19)]{HS}
\begin{gather}\label{eq_resolvent_upper}
    (\Delta_i+k^2)^{-1}(z,z') \le C \left(d(z,z')^{2-N}+d(z,z')^{2-n_i}\right) \textrm{exp}\left(-ck d(z,z')\right), \\ \label{eq_resolvent_derivative}
    |\nabla_z (\Delta_i+k^2)^{-1}(z,z')| \le C \left(d(z,z')^{1-N}+d(z,z')^{1-n_i}\right) \textrm{exp}\left(-ck d(z,z')\right).
\end{gather}

\section{Proof of The Main Result: Harmonic Annihilation}\label{Section_proof}

In this section, we prove Theorem \ref{thm_MainResult}, which serves as a prototype for our further discussion, and present some direct consequences of Theorem~\ref{thm_MainResult}. From now on, sometimes we may use notion $A\lesssim B$ to denote $A\le C B$ for some constant $C>0$ and in the same spirit, one uses $A\gtrsim B$ and $A \simeq B$. We also write $dz = d\mu(z)$ for simplicity.

To better describe the idea of the proof, we consider the special case when $\mathcal{M} = \mathbb{R}^n \# \mathbb{R}^n$ for $n\ge 3$. Let us repeat the argument from \cite[Introduction]{CCH} a little bit. We compactify $\mathcal{M}$ to a compact manifold $\Tilde{\mathcal{M}}$ by adding spheres at both of its ends. We denote the boundaries of ends by $\partial \mathcal{M}_1$ and $\partial \mathcal{M}_2$ (which is also the boundary of $\Tilde{\mathcal{M}}$ i.e. $\partial \Tilde{\mathcal{M}} = \partial \mathcal{M}_1 \cup \partial \mathcal{M}_2$). Then, Carron, Coulhon and Hassell noticed that the kernel of $\Delta^{-1/2}$ has the following expansion when $z'\to \partial \Tilde{\mathcal{M}}$ ($z\in \Tilde{\mathcal{M}}$ fixed)
\begin{equation}\label{eq_CCH}
    \Delta^{-1/2}(z,z') \sim \sum_{j=n-1}^\infty a_j(z) |z'|^{-j}, 
\end{equation}
where $|\cdot|:= \sup_{y\in K}d(\cdot,y)$ and the leading coefficient $a_{n-1}(z)$ is a harmonic function on $\Tilde{\mathcal{M}}$ which tends to 1 on $\partial \mathcal{M}_1$ and to 0 on $\partial \mathcal{M}_2$. The non-vanishing property of $\nabla a_{n-1}$ (maximal principle) makes the decay order of the kernel of the Riesz transform on the right variable, $z'$, to $n-1$. This phenomenon terminates the $L^p$ boundedness of the leading term to $p=n$ (think of it as an individual operator with kernel defined on $(\Tilde{\mathcal{M}}\times \partial \Tilde{\mathcal{M}})$), since the leading term, $|z'|^{1-n}$, can only match with functions in $L^{n-\epsilon}(\partial \Tilde{\mathcal{M}})$ ($\epsilon>0$), and this issue also, in a sense, explains why \eqref{R_p} can not hold beyond the exponent $p=n$. To transit this idea to the reverse inequality, we note that for suitable $f,g$, 
\begin{equation}\label{eq_moti}
    \langle \Delta^{1/2}f, g \rangle = \langle \nabla f, \nabla \Delta^{-1/2}g \rangle.
\end{equation}
Suppose $g$ is supported near the boundary $\partial \Tilde{\mathcal{M}}$. By using the expansion formula above, the leading term of the inner product behaves like
\begin{gather*}
\int_{\Tilde{\mathcal{M}}} \nabla f(z) \cdot \nabla a_{n-1}(z) \int_{\partial \Tilde{\mathcal{M}}} |z'|^{1-n}g(z')dz' dz = \int_{\partial \Tilde{\mathcal{M}}} |z'|^{1-n}g(z') \int_{\Tilde{\mathcal{M}}} f(z)  \Delta a_{n-1}(z) dz dz' = 0
\end{gather*}
since $a_{n-1}$ is harmonic. This implies that the "problematic" leading term vanishes, and the new leading term $\nabla a_n(z)|z'|^{-n}$ has at least a nice decay in the right variable i.e. $|z'|^{-n}\in L^p(\partial \Tilde{\mathcal{M}})$ for all $p\in(1,\infty)$. This observation allows us to expect that \eqref{eq_RRp} may hold beyond the dual range $(n',\infty)$.

\begin{proof}[Proof of Theorem \ref{thm_MainResult}]
By \cite[Theorem 1.2]{HS}, we know that on $\mathcal{M}$ (recall $n_*=\min_j n_j$, the smallest Euclidean dimension among all the ends and $n_*'$ its conjugate exponent)
\begin{equation*}
    \||\nabla f|\|_{p} \le C \|\Delta^{1/2}f\|_{p} \quad \forall 1<p<n_* \quad \forall f\in \mathcal{C}_c^\infty(\mathcal{M}).
\end{equation*}
It then follows by duality \cite[Proposition 2.1]{CD}
\begin{equation*}
    \|\Delta^{1/2}f\|_{p} \le C \||\nabla f|\|_{p} \quad \forall n_*' < p < \infty \quad \forall f\in \mathcal{C}_c^\infty(\mathcal{M}).
\end{equation*}
Therefore, to complete the proof, it suffices to verify \eqref{eq_RRp} for $p\in (1,n_*']$. In what follows, we use a duality argument to prove this inequality. That is we prove 
\begin{gather*}
    \left|\left\langle \Delta^{1/2}f, g \right\rangle \right| \le C\||\nabla f|\|_p \|g\|_{p'}
\end{gather*}
for all $g\in \mathcal{C}_c^\infty(\mathcal{M})$ with $\|g\|_{p'}=1$ and the constant is independent of $g$. First of all, for any $f,g\in \mathcal{C}_c^\infty(\mathcal{M})$, since $\Delta$ is positive and essentially self-adjoint, we have
\begin{equation*}
    \langle \Delta (\Delta+k^2)^{-1}f, g \rangle = \langle f, \Delta (\Delta+k^2)^{-1}g \rangle = \langle \nabla f, \nabla (\Delta+k^2)^{-1}g \rangle \quad \forall k>0.
\end{equation*}
Note that, by spectral theory
\begin{equation*}
    \Delta^{1/2}f =  \frac{2}{\pi} \int_0^\infty \Delta (\Delta+k^2)^{-1}f dk, 
\end{equation*}
where the integral should be understood as the truncated limit: $\lim_{\epsilon \to 0;R\to \infty}\int_{\epsilon}^{R}$, see \cite{AC}. However, in the rest of the proof, we keep using the above notion for simplicity and all estimates are uniform in $\epsilon, R$.

Compare with \eqref{eq_delta-1/2}. We have
\begin{equation*}
    \langle \Delta^{1/2}f, g \rangle = \left\langle \nabla f, \nabla \int_0^\infty (\Delta+k^2)^{-1}g dk \right\rangle,
\end{equation*}
where we omit the constant since it plays no role in the following argument. Follow the ideas from \cite{CCH} and \cite{HS}, the Riesz transform can be divided into the summation of low and high energy parts. That is
\begin{equation*}
    \nabla \int_0^\infty (\Delta+k^2)^{-1} dk = \nabla \int_0^1 (\Delta+k^2)^{-1} dk + \nabla \int_1^\infty (\Delta+k^2)^{-1} dk.
\end{equation*}
Now, recall \cite[Proposition 5.1]{HS}, the high energy part of Riesz transform is well bounded on all $L^q$ spaces i.e.
\begin{equation*}
    \left\| \left|\nabla \int_1^\infty (\Delta+k^2)^{-1}g dk \right| \right\|_q \le C \|g\|_q \quad \forall 1<q<\infty
\end{equation*}
and hence its contribution to the inner product can be estimated by
\begin{equation*}
    \left|\left\langle \nabla f, \nabla \int_1^\infty (\Delta+k^2)^{-1}g dk \right\rangle \right| \le C \||\nabla f|\|_p \|g\|_{p'}. 
\end{equation*}
In the sequel, with the help of parametrix construction introduced in the last section. It suffices to consider the following inner product
\begin{equation*}
    \left\langle \nabla f, \nabla \int_0^1 (\Delta+k^2)^{-1}g dk \right\rangle = \left\langle \nabla f, \sum_{j=1}^4 \int_0^1 \nabla G_j(k)g dk \right\rangle.
\end{equation*}
Notice that by \eqref{eq_G2_bdd} and \eqref{eq_G4_bdd}, we have
\begin{equation*}
    \left\| \left|\int_0^1 \nabla G_j(k)gdk \right| \right\|_{q} \le C\|g\|_{q} \quad \forall 1<q<\infty \quad j=2,4.
\end{equation*}
Hence, by a similar argument for the high energy part, their corresponding inner products are bounded by $C\||\nabla f|\|_p\|g\|_{p'}$. Our next goal is to treat $\nabla G_1$ part. Note that by the explicit formula \eqref{G1}, when the gradient hits $G_1(k)$, it will produce 2 terms for each $i$ which can be written as 
\begin{equation*}
    \nabla (\Delta_i+k^2)^{-1}(z,z') \phi_i(z) \phi_i(z') + (\Delta_i+k^2)^{-1}(z,z') \nabla \phi_i(z) \phi_i(z') := G_{11} + G_{12}.
\end{equation*}
However, the first term acts as a bounded operator on all $L^q$ spaces. Indeed, the first term corresponds to the following operator in the Riesz transform: 
\begin{align*}
    \phi_i \left( \nabla \Delta_{i}^{-1/2}\right)\left(\pi/2-\arctan{\sqrt{\Delta_i}}\right) \phi_i,
\end{align*}
where its boundedness is guranteed by a standard result of the Riesz transform and a multiplier theorem in \cite{CS}. Therefore, for the same reason as before, the inner product corresponding to $G_{11}$ is again bounded by $C\||\nabla f|\|_p\|g\|_{p'}$.

With argument above, we are left to consider the inner product corresponding to $G_{12}+\nabla G_3(k)$. We write down an explicit formula by using \eqref{G3}
\begin{align*}
    \Bigg \langle &\nabla f, \sum_{i=1}^l \int_{\mathcal{M}} \int_0^1 \big[\nabla \phi_i(z)(\Delta_i+k^2)^{-1}(z,z')\phi_i(z')+ \\
    &\nabla u_i(z,k)(\Delta_i+k^2)^{-1}(z_i^0,z')\phi_i(z')\big] dk g(z')dz' \Bigg \rangle_{L^2(\mathcal{M},dz)}.
\end{align*}

Now, the crux is to find the "harmonic part" of each $1\le i\le l$. We proceed with the following strategy, see also \cite{HS}. The first step is to replace $"z"$ in the resolvent factor above by $"z_i^0"$ i.e. for each $1\le i\le l$, we write 
\begin{equation*}
    (\Delta_i+k^2)^{-1}(z,z') = \left[(\Delta_i+k^2)^{-1}(z,z')-(\Delta_i+k^2)^{-1}(z_i^0,z')\right] + (\Delta_i+k^2)^{-1}(z_i^0,z').
\end{equation*}
Observe that $\nabla \phi_i$ is supported in a small "ring" near the connection $K$. We claim that the kernel of the "difference part": 
\begin{gather}\label{eq_difference}
    \int_0^1 \left(\nabla \phi_i(z)[(\Delta_i+k^2)^{-1}(z,z')- (\Delta_i+k^2)^{-1}(z_i^0,z')]\phi_i(z')\right) dk
\end{gather}
acts as a bounded operator on $L^q$ for all $1<q<\infty$. 

\begin{lemma}\label{lemma_difference_Lq}
Under the assumptions of Theorem~\ref{thm_MainResult}, the operator with kernel given by \eqref{eq_difference} is bounded on $L^q$ for all $1<q<\infty$.   
\end{lemma}

\begin{proof}[Proof of Lemma~\ref{lemma_difference_Lq}]
Indeed, let $D_i$ be the support of $\nabla \phi_i$ and $\sigma = \sup_{x\in D_i}d(x,z_i^0)$. Then if $d(z_i^0,z')\le 2 \sigma$, we use \eqref{eq_resolvent_upper} to bound \eqref{eq_difference} by 
\begin{equation}\label{eq_local}
    C \chi_{D_i}(z)\chi_{\{E_i:d(z_i^0,z')\le 2\sigma\}}(z') \left[d(z,z')^{1-N}+d(z_i^0,z')^{1-N}\right].
\end{equation}
It is plain that
\begin{align*}
     \sup_{z\in \mathcal{M}} \int \chi_{D_i}(z)\chi_{\{E_i:d(z_i^0,z')\le 2\sigma\}}(z')\left[d(z,z')^{1-N}+d(z_i^0,z')^{1-N}\right] dz'\\
    \lesssim \int_0^{3\sigma} s^{1-N+N-1} ds
    \lesssim 1.
\end{align*}
On the other hand, a similar estimate gives
\begin{gather*}
    \sup_{z'\in \mathcal{M}} \int \chi_{D_i}(z) \chi_{\{E_i:d(z_i^0,z')\le 2\sigma\}}(z') \left[d(z,z')^{1-N}+d(z_i^0,z')^{1-N}\right] dz\lesssim 1.
\end{gather*}
It then follows by the Schur test, the local part \eqref{eq_local} (i.e. the part where $d(z_i^0,z')\le 2\sigma$) is bounded on $L^q$ for $q\in [1,\infty]$. While for $d(z_i^0,z')\ge 2\sigma$, we use gradient estimate \eqref{eq_resolvent_derivative} to bound \eqref{eq_difference} by
\begin{gather}\label{eq_remote}
     C\chi_{D_i}(z) \chi_{E_i}(z')d(z^*,z')^{-n_i} \lesssim \chi_{D_i}(z) \chi_{E_i}(z')d(z_i^0,z')^{-n_i},
\end{gather}
where $z^*$ is some point seating on the line segment jointing $z$ and $z_i^0$. Now, it is plain that \eqref{eq_remote} acts boundedly on $L^q$ for all $q\in (1,\infty)$ by an $L^q(\mathcal{M},L^{q'}(\mathcal{M}))$ argument. The claim has been verified.   
\end{proof}

We may continue, for the same reason as before, it is enough to treat inner products
\begin{equation*}
    \left\langle \nabla f, \int_{E_i} \int_{0}^1 [\nabla \phi_i(z) + \nabla u_i(z,k)] (\Delta_i+k^2)^{-1}(z_i^0,z')\phi_i(z')dk g(z')dz' \right\rangle \quad \forall 1\le i\le l,
\end{equation*}
where we change the space $\mathcal{M}$ to $E_i$ since the support of $\phi_i$ is contained in the $i$th end entirely.

Then, we split
\begin{equation*}
    \nabla \phi_i(z) + \nabla u_i(z,k) = \nabla [\phi_i(z)+u_i(z,0)] + \nabla [u_i(z,k)-u_i(z,0)] := I+II.
\end{equation*}
Recall that $u_i(z,0) = \lim_{\epsilon \to 0}(\Delta+\epsilon^2)^{-1}v_i(z)$, where $v_i = -\Delta \phi_i$. And $\Phi_i(z) := \phi_i(z) + u_i(z,0)$ is the unique harmonic function which tends to 1 at the infinity of $E_i$ and to 0 at the infinity of all other ends. Set
\begin{equation*}
    \overline{C(g)} = \int_{E_i} \int_{0}^1 (\Delta_i+k^2)^{-1}(z_i^0,z')dk \phi_i(z') g(z')dz'.
\end{equation*}
We pick a smooth bump function $\eta_r$ such that $\eta_r=1$ on some ball $B_r$ with radius $r$ and vanishes outside $B_{2r}$. Moreover, it has property that $\||\nabla \eta_r|\|_\infty \le Cr^{-1}$. It follows by integration by parts (recall $f\in \mathcal{C}_c^\infty(\mathcal{M})$) and dominated convergence theorem, the inner product corresponding to $I$ equals
\begin{gather*}
    C(g) \lim_{r\to \infty} \langle \nabla f, \eta_r \nabla \Phi_i \rangle = C(g) \lim_{r\to \infty} \langle f, \eta_r \textrm{div} (\nabla \Phi_i) \rangle + C(g) \lim_{r\to \infty} \langle f, \nabla \eta_r \cdot \nabla \Phi_i \rangle,
\end{gather*}
where "$\textrm{div}$" is the Riemannian divergence operator. Note that the second term on the RHS is bounded by $C(g)\lim_{r\to \infty}Cr^{-1}\|f\|_1\||\nabla \Phi_i|\|_\infty$ which tends to zero by Lemma \ref{leKey} (i.e. $\nabla \Phi_i$ is uniformly bounded). Hence, we conclude that
\begin{equation}\label{eq_ha}
    C(g) \langle \nabla f, \nabla \Phi_i \rangle = C(g) \langle f, \Delta \Phi_i \rangle = 0 \quad \forall 1\le i\le l.
\end{equation}
Consequently, we only need to consider the following bilinear form ($1\le i\le l$)
\begin{equation}
    \mathcal{I}_i(f,g):= \int_{\mathcal{M}}\nabla f(z) \cdot \int_{E_i} \int_{0}^1 \nabla [u_i(z,k)-u_i(z,0)] (\Delta_i+k^2)^{-1}(z_i^0,z')\phi_i(z')dk g(z')dz' dz.
\end{equation}
Define
\begin{align}\label{T_i}
    \mathcal{T}_ig(z) &= \int_{E_i} \int_{0}^1 \nabla [u_i(z,k)-u_i(z,0)] (\Delta_i+k^2)^{-1}(z_i^0,z')\phi_i(z')dk g(z')dz'\\ \nonumber
    &=\int_0^1 \nabla [u_i(z,k)-u_i(z,0)] \int_{E_i} (\Delta_i+k^2)^{-1}(z_i^0,z')\phi_i(z') g(z')dz' dk. 
\end{align}
Note that, for each $0\le k\le 1$, we have by Lemma \ref{leKey} that the following uniform estimate holds
\begin{equation}\label{eq_uni}
    \| |\nabla [u_i(\cdot,k)-u_i(\cdot,0)]| \|_{L^\infty(\mathcal{M})}\le Ck.
\end{equation}
And meanwhile, a pointwise estimate: for all $z\in E_j$ and $1\le j\le l$
\begin{align}\label{eq_point}
    |\nabla [u_i(z,k)-u_i(z,0)]|&\le |\nabla u_i(z,k)|+ |\nabla u_i(z,0)|\lesssim d(z_j^0,z)^{1-n_j}.
\end{align}
By \eqref{eq_resolvent_upper} and Hölder's inequality, a straightforward estimate gives
\begin{align*}
    \int_{E_i} (\Delta_i+k^2)^{-1}(z_i^0,z')\phi_i(z') g(z')dz' dk &\le \|g\|_{p'} \left(\int_{E_i}\left|(\Delta_i+k^2)^{-1}(z_i^0,z')\phi_i(z')\right|^p dz' \right)^{1/p}\\
    &\lesssim \|g\|_{p'} \left(\int_1^\infty e^{-ckr} r^{n_i-p(n_i-2)-1}dr \right)^{1/p}   \\
    &\lesssim \|g\|_{p'} \begin{cases}
        k^{\frac{n_i}{p'}-2}, & p'>n_i/2\\
        k^{-1/p}, & p'\le n_i/2
    \end{cases}\quad \\
    &:= \|g\|_{p'}F(k).
\end{align*}
There are three situations,
\begin{align*}
        \mathcal{S}_1 = \{i: n_i > p'\}, \quad \mathcal{S}_2 = \{i: n_i = p'\}, \quad
    \mathcal{S}_3 = \{i: n_i < p'\}.
\end{align*}

\begin{lemma}\label{lemma_T_i_pointwise_estimates}
Under the assumptions of Theorem~\ref{thm_MainResult}, we have for $1\le j\le l$,
\begin{align*}
    | \mathcal{T}_ig(z)|\lesssim \|g\|_{p'} \begin{cases}
        d(z_j^0,z)^{1-n_j}, & i\in \mathcal{S}_1, z\in E_j,\\
        d(z_j^0,z)^{1-n_j} (1+|\log{d(z_j^0,z)}|), & i\in \mathcal{S}_2, z\in E_j,\\
        d(z_j^0,z)^{(1-n_j)\frac{n_i}{p'}}, & i\in \mathcal{S}_3, z\in E_j,\\
        1, & 1\le i\le l, z\in K.
    \end{cases}
\end{align*}
\end{lemma}

\begin{proof}[Proof of Lemma~\ref{lemma_T_i_pointwise_estimates}]
\textbf{Case 1.} If $i\in \mathcal{S}_1$, we use pointwise estimate \eqref{eq_point} to get
\begin{align}
    \label{eq_S_1}
    |\mathcal{T}_ig(z)|&\lesssim \|g\|_{p'} \int_0^1 |\nabla\left(u_i(z,k)-u_i(z,0)\right)| |F(k)| dk \\ \nonumber
    &\lesssim  d(z_j^0,z)^{1-n_j} \|g\|_{p'} \int_0^1|F(k)|dk \lesssim d(z_j^0,z)^{1-n_j} \|g\|_{p'} \quad \forall z\in E_j
\end{align}
since $F(k)$ is integrable near 0 when $i\in S_1$.

\textbf{Case 2.} If $i\in \mathcal{S}_2$ i.e. $n_i=p'$. In this situation, we pick some $0<\delta<1$ to be determined later and use uniform estimate \eqref{eq_uni} for interval $(0,\delta)$ and pointwise estimate \eqref{eq_point} for $(\delta,1)$ to obtain for all $z\in E_j$,
\begin{align*}
    | \mathcal{T}_ig(z)|&\lesssim   \|g\|_{p'} \left(\int_0^\delta k^{\frac{n_i}{p'}-1}dk + d(z_j^0,z)^{1-n_j} \int_\delta^1 k^{\frac{n_i}{p'}-2}dk\right)\\
    &\simeq \|g\|_{p'} \left(\int_0^\delta dk + d(z_j^0,z)^{1-n_j} \int_\delta^1 k^{-1}dk\right)\\
    &\lesssim \|g\|_{p'} (\delta + d(z_j^0,z)^{1-n_j} |\log{\delta^{-1}}|).
\end{align*}
By picking $\delta = c d(z_j^0,z)^{1-n_j}$, where $c>0$ is to make sure that $\delta < 1$, we get
\begin{equation}\label{eq_S_2}
    | \mathcal{T}_ig(z)|\lesssim \|g\|_{p'} d(z_j^0,z)^{1-n_j} (1+|\log{d(z_j^0,z)}|) \quad \forall z\in E_j.
\end{equation}

\textbf{Case 3.} If $i\in \mathcal{S}_3$, we use a similar estimates as above to get
\begin{align*}
    | \mathcal{T}_ig(z)|&\lesssim  \|g\|_{p'} \left(\int_0^\delta k^{\frac{n_i}{p'}-1}dk + d(z_j^0,z)^{1-n_j} \int_\delta^1 k^{\frac{n_i}{p'}-2}dk\right)\\
    &\lesssim \|g\|_{p'} (\delta^{\frac{n_i}{p'}} + d(z_j^0,z)^{1-n_j} \delta^{\frac{n_i}{p'}-1}).
\end{align*}
Again, by choosing $\delta = c d(z_j^0,z)^{1-n_j}$, we have
\begin{equation}\label{eq_S_3}
    | \mathcal{T}_ig(z)|\lesssim \|g\|_{p'} d(z_j^0,z)^{(1-n_j)\frac{n_i}{p'}} \quad \forall z\in E_j.
\end{equation}

Finally, for $z\in K$, we simply use the uniform estimate \eqref{eq_uni} for all $1\le i\le l$, 
\begin{equation}\label{eq_S_K}
    | \mathcal{T}_ig(z)| \lesssim  \|g\|_{p'} \int_0^1 kF(k)dk \lesssim \|g\|_{p'} \quad \forall z\in K
\end{equation}
since $kF(k)$ is integrable near 0 for all cases.    
\end{proof}

Now, the remaining of the proof is standard. By Lemma~\ref{lemma_T_i_pointwise_estimates}, for $i\in \mathcal{S}_1$, \eqref{eq_S_1} and \eqref{eq_S_K} imply
\begin{align*}
    \| \mathcal{T}_ig\|_{p'}^{p'} &\lesssim  \int_K \|g\|_{p'}^{p'} dz + C\|g\|_{p'}^{p'} \sum_{j=1}^l \int_{E_j} d(z_j^0,z)^{p'(1-n_j)}dz\\
    &\lesssim  \|g\|_{p'}^{p'} \left(1+ \sum_{j=1}^l \int_1^\infty r^{p'(1-n_j)+n_j-1}dr \right) \lesssim \|g\|_{p'}^{p'}
\end{align*}
since $1<p\le n_*'$ which implies $p' \ge n_*$ and hence $p'>n_j'$ for all $1\le j\le l$. Next, for $i\in \mathcal{S}_2$, we have by \eqref{eq_S_2} and \eqref{eq_S_K}
\begin{align*}
    \| \mathcal{T}_ig\|_{p'}^{p'} &\lesssim  \|g\|_{p'}^{p'} +  \|g\|_{p'}^{p'} \sum_{j=1}^l \int_{E_j} d(z_j^0,z)^{p'(1-n_j)} (1+|\log{d(z_j^0,z)}|)^{p'}    dz\\
    &\lesssim  \|g\|_{p'}^{p'} \left(1+ \sum_{j=1}^l \int_1^\infty r^{p'(1-n_j)+n_j-1}(1+\log{r})^{p'}dr \right)\\
    &\lesssim  \|g\|_{p'}^{p'} + \|g\|_{p'}^{p'} \sum_{j=1}^l \int_0^\infty e^{s[p'(1-n_j)+n_j]} (1+s)^{p'} ds \lesssim  \|g\|_{p'}^{p'}
\end{align*}
since $1<p\le n_*'$ which implies $p' \ge n_*$ and hence $p'>n_j'$ for all $1\le j\le l$. Finally, for $i\in \mathcal{S}_3$, by \eqref{eq_S_3} and \eqref{eq_S_K}
\begin{align*}
    \| \mathcal{T}_ig\|_{p'}^{p'} &\lesssim  \|g\|_{p'}^{p'} +  \|g\|_{p'}^{p'} \sum_{j=1}^l \int_{E_j} d(z_j^0,z)^{(1-n_j)\frac{n_i}{p'} p'}    dz\\
    &\lesssim \|g\|_{p'}^{p'} + \|g\|_{p'}^{p'} \sum_{j=1}^l \int_{1}^\infty r^{(1-n_j)n_i+n_j-1}dr \lesssim  \|g\|_{p'}^{p'}
\end{align*}
since $n_j' < n_i$ for all $1\le i,j\le l$. As a consequence, one confirms that for all $1\le i\le l$,
\begin{equation*}
    \| \mathcal{T}_ig\|_{p'} \lesssim  \|g\|_{p'}.
\end{equation*}
As a consequence,
\begin{align*}
    \sum_{i=1}^l \mathcal{I}_i(f,g) = \sum_{i=1}^l \left\langle \nabla f,  \mathcal{T}_ig \right\rangle \le \sum_{i=1}^l \||\nabla f|\|_p \| \mathcal{T}_ig\|_{p'} \lesssim  \||\nabla f|\|_{p} \|g\|_{p'}.
\end{align*}
Now, we are allowed to assert that for all $f,g\in \mathcal{C}_c^\infty(\mathcal{M})$,
\begin{equation*}
    |\langle \Delta^{1/2}f, g\rangle| \le C \||\nabla f|\|_p \|g\|_{p'} \quad \forall 1<p\le n_*'.
\end{equation*}
By ranging over all functions $g\in \mathcal{C}_c^\infty(\mathcal{M})$ with $\|g\|_{p'}\le 1$, one concludes by duality that
\begin{equation*}
    \|\Delta^{1/2}f\|_p \le  C\||\nabla f|\|_p \quad \forall 1<p\le n_*'
\end{equation*}
which completes the proof of Theorem \ref{thm_MainResult}.

\end{proof}

\begin{remark}
Here we introduce an alternative proof for Theorem \ref{thm_MainResult} but under an additional assumption that $n_*>4$. We consider the difference term in the following way:
\begin{gather*}
\nabla [u_i(z,k)-u_i(z,0)] = \nabla\left[(\Delta+k^2)^{-1}-(\Delta+0^2)^{-1}\right]v_i(z) \\
= - \int_0^{k^2} \nabla(\Delta+s)^{-2}v_i(z) ds = -2 \int_0^k \nabla(\Delta+\lambda^2)^{-2}v_i(z) \lambda d\lambda.
\end{gather*}
Before moving on, we need to check the definition of the higher order resolvent term. By \cite[Proposition 3.2]{BS}, comparing to Lemma \ref{leKey}, we know that if $v_i\in C_c^\infty$, we have for $0\le m<n_*/2-1$, $0\le \lambda \le 1$, the higher order resolvent $(\Delta+\lambda^2)^{-1-m}v_i$ is well-defined for all $\lambda \in [0,1]$ and for all $z\in E_j$,
\begin{gather*}
    |\nabla (\Delta+\lambda^2)^{-1-m}v_i(z)| \lesssim \|v_i\|_\infty d(z_j^0,z)^{2m+1-n_j} \textrm{exp}\left(-c\lambda d(z_j^0, z)\right).
\end{gather*}
In our case, take $m=1$ (then we need $n_*>4$) we get for all $z\in E_j$,
\begin{align*}
    |\nabla (u_i(z,k)-u_i(z,0))| &\lesssim \int_0^k d(z_j^0,z)^{3-n_j} \textrm{exp}\left(-c\lambda d(z_j^0, z)\right) \lambda d\lambda \lesssim  k d(z_j^0,z)^{2-n_j}.
\end{align*}
Now, for all $z\in E_j$,
\begin{gather*}
    |\mathcal{T}_ig(z)| \lesssim  d(z_j^0,z)^{2-n_j} \|g\|_{p'} \int_0^1 k|F(k)|dk \lesssim   d(z_j^0,z)^{2-n_j} \|g\|_{p'}
\end{gather*}
and hence, $\|\mathcal{T}_ig\|_{p'} \lesssim \|g\|_{p'}$ as long as $p<n_j/2$ for all $1\le j\le l$ i.e. $p<n_*/2$. Recall that since $n_*>4$, so we have $1<p\le n_*' < 2 < n_*/2 \le n_j/2$. We therefore complete the proof under the further condition $n_*>4$. This method simplifies the calculation in the proof of Theorem \ref{thm_MainResult} but as a compensation, the result does not include the case when $n_*=3,4$.  
\end{remark}

\begin{remark}\label{remark}
It is also worth discussing some "drawback" of this "harmonic annihilation" method. We rely on three critical estimates. One is the convergence of the limit (in distribution sense at least)
\begin{equation}\label{eq_green}
    \lim_{k\to 0}(\Delta+k^2)^{-1}(z,z'), \quad z'\in K.
\end{equation}
The second one is the uniform estimate \eqref{eq_uni} which gives an extra decay of "$k$" near zero. Finally, the basic assumption: $f\in \mathcal{C}_c^\infty(\mathcal{M})$. For the last one, one checks that the key step in the proof, \eqref{eq_ha}, fails to hold if $f$ does not vanish at infinity since the harmonic function $\Phi_i$ tends to 1 at the infinity of $E_i$. That is, by our proof, \eqref{eq_RRp} does not extend trivially to $f\in \Dot{W}^{1,p}$ as in \cite{AC}, where \eqref{Doubling} and $(\textrm{P}_2)$ are assumed.

Let us consider a critical case studied in \cite{HNS} (the formal proof of this case, i.e. Theorem~\ref{thm_RR_critical} will be postponed to Section~\ref{section_6})
\begin{equation*}
    \mathcal{M}=(\mathbb{R}^2\times \mathcal{M}_-)\#(\mathbb{R}^{n_+}\times \mathcal{M}_+), \quad n_+\ge 3,
\end{equation*}
where $\mathcal{M}_{-}, \mathcal{M}_+$ are some compact manifolds. It was showed in \cite{HNS} that $\nabla \Delta^{-1/2}$ on this manifold is bounded on $L^p$ if and only if $1<p\le 2$. It then follows by duality, \eqref{eq_RRp} holds immediately for $p\ge 2$. While for $1<p<2$, the same idea giving in the proof of Theorem~\ref{thm_MainResult} does not work directly. Indeed, to estimate operator $T_i$, defined by \eqref{T_i}, we need both some decay on "$k$" (near 0) and "$d(z_i^0,z)$" (at infinity of each end $E_i$). Especially, the additional decay of "$k$" plays a crucial role in the convergence of the integral. However, in this critical case (some ends have "dimension" 2), one of the most significant differences is that the Green's function does not coincide with the limit \eqref{eq_green}. In fact, there is a logarithmic increase, $\log{k^{-1}}$, for \eqref{eq_green} when $k\to 0$ (also, for this reason, the parametrix construction in \cite{HNS} and \cite{HS} are technically different). To handle this problem, authors in \cite{HNS} introduced a factor $\textrm{ilg}(k)$ which annihilates the "log" term as $k\to 0$. Formally,
\begin{equation*}
    \textrm{ilg}(k) = \begin{cases}
        \frac{-1}{\log{k}}, & 0<k\le 1/2,\\
        0, & k=0,
    \end{cases}
\end{equation*}
which approaches to 0 slower than any polynomial with positive power. This manipulation guarantees the convergence of the limit in \eqref{eq_green} (\cite[Lemma 2.14]{HNS}) but as a compensation, the uniform estimate \eqref{eq_uni} becomes
\begin{equation*}
    \|\nabla [u(\cdot,k)-u(\cdot,0)]\|_\infty \lesssim \textrm{ilg}(k) + O(k).
\end{equation*}
Now, if we apply the same strategy as in the proof of Theorem~\ref{thm_MainResult}, one ends up with estimate: for $z\in K$,
\begin{equation*}
    |\mathcal{T}_i(g)(z)|\lesssim \|g\|_{p'} \int_0^{k_0} \textrm{ilg}(k) k^{-2/p} dk \quad 1<p<2,
\end{equation*}
where $k_0\in (0,1/2)$ is determined by the parametrix construction, see \cite[Section 3]{HNS}. Unfortunately, by the slow decay of $\textrm{ilg}(k)$, the above integral diverges.
\end{remark}

We end this section by giving some direct consequences of Theorem \ref{thm_MainResult}.
\begin{corollary}
Let $\mathcal{M}$ be defined by \eqref{eq_Manifold}, then \eqref{E_p} holds if and only if $1<p<n_*$.
\end{corollary}

\begin{proof}
It follows by combining the results of \cite[Theorem 1.2]{HS} and Theorem \ref{thm_MainResult}.
\end{proof}

\begin{corollary}\label{cor_hodge}
Let $\mathcal{M}$ be defined by \eqref{eq_Manifold}, then the Hodge projector $d_{\mathcal{M}}\Delta^{-1}d_{\mathcal{M}}^*$ is bounded on $L^p$ if and only if $n_*'<p< n_*$, where $d_{\mathcal{M}}$ is the exterior derivative and $d_{\mathcal{M}}^*$ its formal adjoint. 
\end{corollary}

\begin{proof}
The range of boundedness follows directly by \cite[Theorem 1.2]{HS} and a duality argument since 
\begin{equation*}
    d_{\mathcal{M}}\Delta^{-1}d_{\mathcal{M}}^* = d_{\mathcal{M}}\Delta^{-1/2}(d_{\mathcal{M}}\Delta^{-1/2})^*.
\end{equation*}
Regarding the negative part, we let $p\ge n_*$. By \cite[Lemma 0.1]{AC}, we know that the validity of $(\textrm{RR}_{p'})$ and the boundedness of $d_{\mathcal{M}}\Delta^{-1}d_{\mathcal{M}}^*$ on $L^p$ implies \eqref{R_p}. However, by \cite[Theorem 1.2]{HS}, we know that \eqref{R_p} is false for all $p\ge n_*$ which implies the unboundedness of $d_{\mathcal{M}}\Delta^{-1}d_{\mathcal{M}}^*$ on $L^p$ since by Theorem~\ref{thm_MainResult} $(RR_{p'})$ holds for all $p'$. And the case $p\le n_*'$ follows by duality since Hodge projector is self-adjoint.
\end{proof}

\begin{corollary}
Let $\mathcal{M}=\mathbb{R}^n \# \mathbb{R}^n$, then the following version Sobolev inequality holds
\begin{equation*}
    \|f\|_q \lesssim \||\nabla f|\|_p \quad \forall f\in \mathcal{C}_c^\infty(\mathcal{M}),
\end{equation*}
where $1/p - 1/q = 1/n$ and $1<p<n$.
\end{corollary}

\begin{proof}
It is well-known that by \cite{GS},
\begin{equation*}
    \|e^{-t\Delta}\|_{1\to \infty} \lesssim  t^{-n/2}.
\end{equation*}
Then it follows by a result of Varopoulos \cite{Varo},  
\begin{equation*}
    \|\Delta^{-1/2}f\|_q \lesssim \|f\|_p \quad 1<p<n \quad 1/p-1/q=1/n.
\end{equation*}
Use Theorem \ref{thm_MainResult} and \cite[Lemma~1]{Russ}, one concludes that
\begin{equation*}
    \|f\|_q \lesssim \|\Delta^{1/2}f\|_p \lesssim \||\nabla f|\|_p.
\end{equation*}

\end{proof}

\begin{remark}
In fact, the case, where $p=1$, i.e. $\|f\|_{\frac{n}{n-1}}\lesssim \||\nabla f|\|_1$ also holds in this setting and it is related to the isoperimetric inequality $|\Omega|^{\frac{n-1}{n}}\lesssim |\partial \Omega|$ for all bounded $\Omega \subset \mathcal{M}$ with smooth boundary. This is studied in the author's PhD thesis.
\end{remark}

\section{Reverse Riesz on Broken Line}\label{Section_RR_Broken_Line}

In this section, we extend the results above to a one-dimensional model studied in \cite{HS2}. For some $d>1$, we consider measure space $(\Tilde{\mathbb{R}},d\mu)$, where $\Tilde{\mathbb{R}} = (-\infty,-1]\cup [1,\infty)$ with measure
\begin{equation}\label{eq_brokenline}
    d\mu(r) = |r|^{d-1}dr,
\end{equation}
where $dr$ is Lebesgue measure and the measure $d\mu$ mimics the radial part of Lebesgue measure in $\mathbb{R}^d$ (when $d$ is an integer). We identify $-1$ and $1$ in the following sense. We say $f\in \mathcal{C}^1(\Tilde{\mathbb{R}})$ if $f(-1)=f(1)$, $f'(-1)=f'(1)$ and it is continuously differentiable. Similarly, we define $\mathcal{C}_c^\infty(\Tilde{\mathbb{R}})$ to be the space of smooth functions with compact support and satisfy equation $f^{(k)}(-1) = f^{(k)}(1)$ for all $k\in \mathbb{N}$. Also, denote by $S_0(\Tilde{\mathbb{R}})$, a subspace of $\mathcal{C}_c^\infty(\Tilde{\mathbb{R}})$ with an additional condition $f(\pm 1)=0$. Apparently, $S_0$ is dense in $L^p(\Tilde{\mathbb{R}})$ for all $p\in [1,\infty)$.

For any $f,g\in \mathcal{C}^1(\Tilde{\mathbb{R}})$, consider quadratic form
\begin{equation*}
    Q(f,g) = \int_{\Tilde{\mathbb{R}}} f'(r) g'(r) d\mu(r).
\end{equation*}
Followed by Friedrichs extension, there exists a unique self-adjoint operator associated to this form
\begin{equation*}
    \Delta_d f = - \frac{d^2}{dr^2}f - \frac{d-1}{|r|} \frac{d}{dr}f \quad f\in \mathcal{C}^1(\Tilde{\mathbb{R}}),
\end{equation*}
and then the Riesz transform (See \cite{HS2})
\begin{equation}\label{eq_formula}
    \nabla \Delta_d^{-1/2} = \nabla \int_0^\infty (\Delta_d+\lambda^2)^{-1}d\lambda,
\end{equation}
where $\nabla$ is the usual derivative i.e. $\nabla f = f'$.

This model usually plays a role in predicting the behaviors of the Riesz transform in higher dimensional cases. Note that when $d\ge 3$ (integer), the operator $\Delta_d$ models the spherical symmetric part of the Laplace-Beltrami operator on $\mathbb{R}^d \# \mathbb{R}^d$ introduced previously. In fact, if we allow the "dimensions" on each end to be different (discussions about such modified one-dimensional model can be found in \cite{N}), by this we mean the measure changes to $|r|^{d_1-1}dr$ for $r\le -1$ and $r^{d_2-1}dr$ for $r\ge 1$ and $d_1\ne d_2$, then the model will grab the behaviour of the radial part of Laplacian on $\mathcal{M}$ given by \eqref{eq_Manifold} with $l=2$. One also can replace the space $(\Tilde{\mathbb{R}},d\mu)$ to $([1,\infty),r^{d-1}dr)$ and consider Dirichlet Laplacian. In \cite{HS2}, Hassell and Sikora proved that the Riesz transform associated to this Dirichlet Laplacian is bounded on $L^p$ if and only if $1<p<d$. Subsequently, Killip, Visan, and Zhang \cite{KVZ} proved that the Riesz transform corresponding to the Dirichlet Laplacian outside a convex domain in $\mathbb{R}^d$ ($d\geq 3$) is bounded on $L^p(\mathbb{R}^d)$ for $p$ in the same range. Additionally, in \cite{JL}, Jiang and Lin proved that for all $d\ge 2$, such Dirichlet Laplacian satisfies \eqref{eq_RRp} for all $p\in (1,\infty)$.

An advantage to consider only the radial part of the Laplacian is that the "dimension" can vary continuously. In this note, we consider all $d>1$. This may allow us to see how our method behaves as the decrease of the dimension from $d\ge 3$ to $1<d\le 2$.

The following is the main result of \cite{HS2}.

\begin{theorem}\cite[Theorem 1.1]{HS2}\label{thm_HS2}
Let $(\Tilde{\mathbb{R}},d\mu)$ be the measure space defined by \eqref{eq_brokenline}. The Riesz transform $\nabla \Delta_d^{-1/2}$ is bounded on $L^p(\Tilde{\mathbb{R}},d\mu)$ if and only if

$(\romannumeral1)$ $1<p<d$ for $d>2$,

$(\romannumeral2)$ $1<p\le 2$ for $d=2$,

$(\romannumeral3)$ $1<p<d'$ for $1<d<2$.
\end{theorem}

We start our discussion with the following toy type result. 

\begin{theorem}[Toy Type]\label{thm_d>2}
Let $d>2$ and $(\Tilde{\mathbb{R}},d\mu)$ be the measure space defined by \eqref{eq_brokenline}. Then, we have 
\begin{gather}\label{eq_RR_d>2}
     \|\Delta_d^{1/2}f\|_p \lesssim \|f'\|_p \quad \forall f\in S_0(\Tilde{\mathbb{R}}),
\end{gather}
where  $\max \left(\frac{d}{2d-3},1 \right)<p<\infty$.
\end{theorem}

\begin{remark}
Note that instead of assuming $f\in C_c^\infty(\Tilde{\mathbb{R}})$, we let $f\in S_0(\Tilde{\mathbb{R}})$ here. The main reason is that the function $k(\cdot)$ in the resolvent construction (see Subsection~\ref{subsection_4.1}) is not in the initial domain of $\Delta_d$. If we let $f\in \mathcal{C}_c^\infty(\Tilde{\mathbb{R}})$, then there will appear some extra non-zero boundary terms when we apply "harmonic annihilation", see the proof of Theorem~\ref{thm_d>2} later.  
\end{remark}

\begin{remark}
Note that $d/(2d-3)=1$ if $d=3$ and equals 2 when $d=2$. Therefore, Theorem~\ref{thm_d>2} implies that when $2<d<3$, \eqref{eq_RRp} holds for $p>d/(2d-3)$ and when $d\ge 3$, \eqref{eq_RRp} holds for all $1<p<\infty$. Particularly, in the context of $d\ge 3$, this result is consistent with the result obtained in Theorem \ref{thm_MainResult}. While for $2<d<3$, the range of boundedness shrinks to $[2,\infty)$ continuously as $d\to 2$.
\end{remark}

\subsection{Asymptotics for The Modified Bessell Functions}\label{subsection_4.1}

As indicated in \eqref{eq_formula}, the method is based on the estimates of the resolvent operator. And it has been computed explicitly in \cite{HS2,N}, the kernel for the resolvent operator has following formula (here we use the version in \cite{N}). For $\lambda>0$, $y\ge 1$
\begin{equation*}
    K_{(\Delta_d+\lambda^2)^{-1}}(x,y) = 
    \begin{cases}
    A(\lambda) k(\lambda|x|) k(\lambda|y|), & x\le -1\\
    B(\lambda) k(\lambda|y|) k(\lambda|x|) + v \lambda^{d-2}k(\lambda|y|)l(\lambda|x|), & 1\le x\le y\\
    B(\lambda) k(\lambda|y|) k(\lambda|x|) + v \lambda^{d-2}l(\lambda|y|) k(\lambda|x|), & x\ge y
    \end{cases}
\end{equation*}
and $K_{(\Delta_d+\lambda^2)^{-1}}(x,y) = K_{(\Delta_d+\lambda^2)^{-1}}(-x,-y)$ if $y\le -1$ where $\nu$ is some constant depending on $d$ and
\begin{gather*}
    A(\lambda) = \frac{-1}{\lambda [k(\lambda) k(\lambda)]'}\quad B(\lambda) = \frac{-v \lambda^{d-2}[k(\lambda) l(\lambda)]'}{[k(\lambda) k(\lambda)]'}.
\end{gather*}
Moreover, $k(r)=r^{1-d/2}K_{|d/2-1|}(r)$ and $l(r)=r^{1-d/2}I_{d/2-1}(r)$, where $K,I$ are the modified Bessell functions which solve the differential equation
\begin{equation*}
    r^2 F'' + rF' = (r^2+(d/2-1)^2)F
\end{equation*}
and hence $k(\lambda r)$ and $l(\lambda r)$ solve the equation (see \cite[Section 3.1]{HS2} for details)
\begin{equation}\label{eq_DE}
    f'' + \frac{d-1}{r}f' = \lambda^2 f.
\end{equation}
Recall estimates, see \cite{AS} or \cite{HS2}

$\bullet$ \textbf{If $1<d<2$},
\begin{align*}
    &k(r)\simeq 
    \begin{cases}
    1, & r<1,\\
    r^{(1-d)/2}e^{-r}, & r\ge 1,
    \end{cases}
    \quad
    &&l(r)\simeq 
    \begin{cases}
    1, & r<1,\\
    r^{(1-d)/2}e^r, & r\ge 1,
    \end{cases}
    \\
    &k'(r)\simeq 
    \begin{cases}
    -r^{1-d}, & r<1,\\
    -r^{(1-d)/2}e^{-r}, & r\ge 1,
    \end{cases}
    \quad
    &&l'(r)\simeq 
    \begin{cases}
    r, & r<1,\\
    r^{(1-d)/2}e^r, & r\ge 1.
    \end{cases}
\end{align*}

$\bullet$ \textbf{If $d=2$},
\begin{align*}
    &k(r)\simeq 
    \begin{cases}
    1-\log \lambda, & r<1,\\
    r^{-1/2}e^{-r}, & r\ge 1,
    \end{cases}
    \quad
    &&l(r)\simeq 
    \begin{cases}
    1, & r<1,\\
    r^{-1/2}e^r, & r\ge 1,
    \end{cases}
    \\
    &k'(r)\simeq 
    \begin{cases}
    -r^{-1}, & r<1,\\
    -r^{-1/2}e^{-r}, & r\ge 1,
    \end{cases}
    \quad
    &&l'(r)\simeq 
    \begin{cases}
    r, & r<1,\\
    r^{-1/2}e^r, & r\ge 1.
    \end{cases}
\end{align*}

$\bullet$ \textbf{If $d>2$},
\begin{align*}
    k(r)&\simeq 
    \begin{cases}
    r^{2-d}, & r<1,\\
    r^{(1-d)/2}e^{-r}, & r\ge 1.
    \end{cases}
    \quad
    &&l(r)\simeq 
    \begin{cases}
    1, & r<1,\\
    r^{(1-d)/2}e^r, & r\ge 1.
    \end{cases}
    \\
    k'(r)&\simeq 
    \begin{cases}
    -r^{1-d}, & r<1,\\
    -r^{(1-d)/2}e^{-r}, & r\ge 1.
    \end{cases}
    \quad
    &&l'(r)\simeq 
    \begin{cases}
    r, & r<1,\\
    r^{(1-d)/2}e^r, & r\ge 1.
    \end{cases}
\end{align*}
Therefore, a direct computation implies that for $\lambda \le 1$, 
\begin{equation}\label{eq_AB}
    |A(\lambda)|, |B(\lambda)| \lesssim \begin{cases}
        \lambda^{d-2}, & 1<d<2,\\
        \frac{-1}{\log{\lambda}}, & d=2,\\
        \lambda^{2d-4}, & d>2.
    \end{cases} 
\end{equation}

\subsection{Proof of Theorem \ref{thm_d>2}}
To better describe the behaviour of the resolvent kernel, as did in \cite{HS2}, we split the kernel into "$kk$" and "$kl$" parts. And "$kk$" part only consists of functions in the form $k(\cdot)k(\cdot)$ and "$kl$" part consists of the rest i.e. functions in the form $k(\cdot)l(\cdot)$. We denote their corresponding resolvent by $(\Delta_d+\lambda^2)^{-1}_{kk}$ and $(\Delta_d+\lambda^2)^{-1}_{kl}$. Obviously
\begin{equation*}
    (\Delta_d+\lambda^2)^{-1} = (\Delta_d+\lambda^2)^{-1}_{kk} + (\Delta_d+\lambda^2)^{-1}_{kl}.
\end{equation*}
\begin{proof}[Proof of Theorem \ref{thm_d>2}]
It follows by Theorem~\ref{thm_HS2} and duality, we only consider the case $1<p\le d'$ when $d\ge 3$ and $d/(2d-3)<p\le d'$ when $2<d<3$. We use resolution to identity
\begin{equation*}
    \Delta_d^{1/2} = \int_0^\infty \Delta_d (\Delta_d+\lambda^2)^{-1}d\lambda.
\end{equation*}
Now, for any $f\in S_0(\Tilde{\mathbb{R}})$, $g\in \mathcal{C}_c^\infty(\Tilde{\mathbb{R}})$, it follows by the self-adjointness of $\Delta_d$ 
\begin{align*}
    \langle \Delta_d^{1/2}f ,g \rangle &= \left \langle \int_0^\infty \Delta_d (\Delta_d+\lambda^2)^{-1}f d\lambda, g \right \rangle = \left\langle \nabla f, \nabla \int_0^\infty (\Delta_d+\lambda^2)^{-1}g d\lambda \right\rangle\\
    &= \left\langle \nabla f, \nabla \int_0^\infty (\Delta_d+\lambda^2)_{kk}^{-1}g d\lambda \right\rangle + \left\langle \nabla f, \nabla \int_0^\infty (\Delta_d+\lambda^2)_{kl}^{-1}g d\lambda \right\rangle.
\end{align*}
By \cite[Theorem 5.1, Remark 5.2]{HS2}, the "$kl$" part of Riesz transform is bounded on $L^q$ for all $1<q<\infty$. It is then clear that
\begin{equation*}
    \left|\left\langle \nabla f, \nabla \int_0^\infty (\Delta+\lambda^2)_{kl}^{-1}g d\lambda \right\rangle \right| \le C\|f'\|_p \|g\|_{p'}. 
\end{equation*}
And hence it is sufficient to consider the "$kk$" part of the inner product. With notation $F(\cdot)=A(\cdot)$ or $B(\cdot)$, we write
\begin{align*}
    \nabla \int_0^\infty (\Delta+\lambda^2)_{kk}^{-1}g(x)d\lambda &= \int_{\Tilde{\mathbb{R}}} \int_0^\infty \lambda k(\lambda|y|)k'(\lambda|x|)F(\lambda)d\lambda g(y)d\mu(y) \\
    &= \int_{\Tilde{\mathbb{R}}} \int_{1}^\infty  + \int_{\Tilde{\mathbb{R}}} \int_0^{1} := T_h(g) + \int_{\Tilde{\mathbb{R}}} \int_0^{1}.
\end{align*}
Use the fact from \cite[Theorem 5.5]{HS2}, the "high energy" part $T_h$ is again bounded on all $L^q$ (in fact, in \cite[Theorem 5.5]{HS2}, authors consider the integral of the high energy part with lower entry $1/\min(|x|,|y|)$ rather than 1, but since $|x|,|y|\ge 1$, the boundedness of $T_h$ directly follows by their result). Therefore,
\begin{equation*}
    |\langle \nabla f, T_h(g) \rangle |\le C\|f'\|_p \|g\|_{p'}.
\end{equation*}

\begin{remark}\label{remark_all_good}
Here, for later use, we make a remark that up to now, all the argument above also works for the case $1<d\le 2$.
\end{remark}

It remains to consider the following bilinear form
\begin{align}\label{eq_bilinear}
    \mathcal{B}(f,g) &= \left\langle \nabla f, \nabla \int_{\Tilde{\mathbb{R}}} \int_0^{1}k(\lambda|y|)k(\lambda|x|)F(\lambda)d\lambda g(y)d\mu(y) \right\rangle\\ \nonumber
    &= \int_{\Tilde{\mathbb{R}}} f'(x) \int_0^1 \left[\frac{d}{dx}k(\lambda |x|)\right] \int_{\Tilde{\mathbb{R}}} k(\lambda|y|)g(y)d\mu(y) F(\lambda) d\lambda d\mu(x).
\end{align}
As $\Delta_d$ mimics the radial part of Laplacian on $\mathbb{R}^d \# \mathbb{R}^d$ (see \cite[Remark~5.7]{HS2}). We may expect the kernel of $\Delta_d^{-1/2}$ enjoys a similar asymptotic formula to \eqref{eq_CCH} when $|y|\to \infty$. Since the kernel of the right entry in the above inner product is almost separated with the integral with respect to $\lambda$ disregarded, the harmonic leading coefficient is probably hidden in the term $k(\lambda |x|)$.

Define $k_\lambda(x) = k(\lambda |x|)$ then set $u(x,\lambda):= \lim_{\epsilon \to 0} (\Delta_d+\epsilon^2)^{-1}k_\lambda(x)$. The existence of function $u(\cdot,\lambda)$ is guarenteed by \cite[Lemma 3.2]{CCH} or \cite[Section 2.3.3]{N} since $k_\lambda(\cdot)$ has rapid decay on ${\Tilde{\mathbb{R}}}$. We then decompose
\begin{equation*}
    k(\lambda|x|) = \left[ k(\lambda|x|) + \lambda^2 u(x,\lambda) \right] - \lambda^2 u(x,\lambda).
\end{equation*}
Recall that $k_\lambda$ is the solution to the ordinary differential equation \eqref{eq_DE}, i.e.
\begin{equation*}
    f''(r) + \frac{d-1}{|r|}f'(r) = \lambda^2 f(r).
\end{equation*}
Hence, the first term in the decomposition is "harmonic" in the sense
\begin{equation*}
    \left(-\frac{d^2}{dx^2}-\frac{d-1}{|x|}\frac{d}{dx}\right)[k_\lambda(\cdot)]+ \Delta_d [\lambda^2 u(\cdot,\lambda)] = 0 \quad \textrm{for} \quad a.e. \quad x\in {\Tilde{\mathbb{R}}}.
\end{equation*}
Consequently, since $f(\pm 1) = 0$, the boundary term vanishes when we apply integration by parts and the "harmonic" part is annihilated:
\begin{gather*}
    \left\langle \nabla f, \int_0^1 \nabla [k(\lambda|x|)+\lambda^2 u(x,\lambda)] \int_{\Tilde{\mathbb{R}}} k(\lambda|y|)g(y)d\mu(y) F(\lambda) d\lambda \right\rangle\\
    = \int_0^1 \left\langle f, \left(-\frac{d^2}{dx^2}-\frac{d-1}{|x|}\frac{d}{dx}\right)[k_\lambda(.)]+ \Delta_d [\lambda^2 u(.,\lambda)] \right\rangle \int_{\Tilde{\mathbb{R}}} k(\lambda|y|)g(y)d\mu(y) F(\lambda) d\lambda = 0.
\end{gather*}
It turns out that one only needs to investigate the following
\begin{equation}\label{eq_innerproduct}
    \mathcal{J}(f,g) = \int_{\Tilde{\mathbb{R}}} \nabla f(x) \int_0^1 \nabla u(x,\lambda) \lambda^2F(\lambda) \left( \int_{\Tilde{\mathbb{R}}} k(\lambda|y|)  g(y)d\mu(y) \right) d\lambda d\mu(x).
\end{equation}

Before moving on, we need some estimates for $|\nabla u(x,\lambda)|$.

\begin{lemma}\label{lemma_du}
Under the assumptions of Theorem~\ref{thm_d>2}, let $u(x,\lambda)$ be defined as above. Then, for all $0\le \delta\le 1$ and $d>2$,
\begin{equation}\label{eq_Du}
    |\nabla u(x,\lambda)| \lesssim \begin{cases}
        \lambda^{1-d-\delta} x^{2-d-\delta}, & d\ge 3\\
        \lambda^{-2} x^{2-d-\delta}, & 2<d<3, \quad 0\le \delta \le 3-d\\
        \lambda^{1-d-\delta} x^{2-d-\delta}, & 2<d<3, \quad 3-d\le \delta \le 1.
    \end{cases}
\end{equation}
\end{lemma}

\begin{proof}[Proof of Lemma~\ref{lemma_du}]
By the formula of the kernel of resolvent and the symmetry of the measure space, we may assume $0<\epsilon<\lambda<1\le x <\epsilon^{-1}$ and $y\ge 1$, and only consider the following quantities with $\epsilon \to0$
\begin{gather*}
    I = \epsilon |F(\epsilon)| |k'(\epsilon x)| \int_1^\infty k(\epsilon y) k(\lambda y) y^{d-1} dy
\end{gather*}
and
\begin{align*}
    II &= \epsilon^{d-1} |k'(\epsilon x)| \int_1^x  l(\epsilon y) k(\lambda y) y^{d-1}dy\\
    &+ \epsilon^{d-1} l'(\epsilon x) \int_x^\infty k(\epsilon y) k(\lambda y) y^{d-1}dy\\
    &:= II_1+II_2. 
\end{align*}
Obviously, $|\nabla u(x,\lambda)|\lesssim I+II$. For $I$, we have
\begin{align*}
 \int_1^\infty k(\epsilon y) k(\lambda y) y^{d-1} dy &= \int_1^{\lambda^{-1}} + \int_{\lambda^{-1}}^{\epsilon^{-1}} + \int_{\epsilon^{-1}}^\infty\\
 &:= I_1 + I_2 + I_3.
\end{align*}
Recall \eqref{eq_AB}, $|F(\epsilon)|\lesssim \epsilon^{2d-4}$ for $\epsilon$ small, hence 
\begin{align*}
    I_1 &\lesssim \epsilon^{2-d} \lambda^{2-d} \int_1^{\lambda^{-1}} y^{3-d} dy \\
    &\lesssim \begin{cases}
        \epsilon^{2-d} \lambda^{2-d}, & d>4\\
        \epsilon^{-2} \lambda^{-2} \log(\lambda^{-1}), & d=4\\
        \epsilon^{2-d} \lambda^{-2}, & 2<d<4
    \end{cases} := \epsilon^{2-d} \mathcal{G}(d,\lambda)
\end{align*}
and similarly,
\begin{align*}
    I_2 &\lesssim \epsilon^{2-d} \lambda^{-2} \int_1^{\lambda/\epsilon} e^{-s} s^{\frac{3-d}{2}} ds \lesssim \epsilon^{2-d} \lambda^{-2},\\
    I_3 &\lesssim \frac{\epsilon^{\frac{1-d}{2}}\lambda^{\frac{1-d}{2}}}{\lambda+\epsilon} e^{-\frac{\epsilon+\lambda}{\epsilon}} \to0 \quad \textit{as} \quad \epsilon \to 0.
\end{align*}
Since $I_1+I_2\lesssim I_1$ and $|k'(\epsilon x)|\lesssim \epsilon^{1-d}x^{1-d}$, one gets
\begin{equation*}
    I\lesssim \epsilon^{d-1} |k'(\epsilon x)| \mathcal{G}(d,\lambda) \lesssim x^{1-d} \mathcal{G}(d,\lambda).
\end{equation*}
For $II$, it is clear that
\begin{gather*}
    II_1\lesssim x^{1-d} \int_1^x |k(\lambda y)| y^{d-1}dy,
\end{gather*}
and
\begin{align*}
    II_2\lesssim \epsilon^2 x \int_x^{\epsilon^{-1}} |k(\lambda y)| ydy + \epsilon^{\frac{d-1}{2}}\lambda^{\frac{1-d}{2}}x \to 0 \quad \textit{as} \quad \epsilon \to 0.
\end{align*}
Therefore, we only need to consider $II_1$. Note that if $x\le \lambda^{-1}$, one has
\begin{gather*}
    \int_1^x |k(\lambda y)|y^{d-1}dy \lesssim \int_1^x \lambda^{2-d} ydy \lesssim \lambda^{2-d} x^2.
\end{gather*}
While for $x>\lambda^{-1}$, 
\begin{gather*}
    \int_1^x |k(\lambda y)|y^{d-1}dy \lesssim \int_1^{\lambda^{-1}} \lambda^{2-d}ydy + \int_{\lambda^{-1}}^x \lambda^{\frac{1-d}{2}} y^{\frac{d-1}{2}} e^{-\lambda y}dy \lesssim \lambda^{-d}.
\end{gather*}
Consequently
\begin{equation*}
    II_1 \lesssim \begin{cases}
        \lambda^{2-d}x^{3-d}, & x\le \lambda^{-1},\\
        \lambda^{-d}x^{1-d}, & x\ge \lambda^{-1}.
    \end{cases}
\end{equation*}
Pick $1\le \delta'\le 2$, we obtain
\begin{align*}
    II_1 \lesssim \begin{cases}
        \lambda^{2-d}x^{3-d} = (\lambda x)^{\delta'} \lambda^{2-\delta'-d}x^{3-\delta'-d} \le \lambda^{2-\delta'-d}x^{3-\delta'-d}, & x\le \lambda^{-1},\\
        \lambda^{-d}x^{1-d}= (\lambda x)^{\delta'-2} \lambda^{2-\delta'-d}x^{3-\delta'-d}\le \lambda^{2-\delta'-d}x^{3-\delta'-d}, & x\ge \lambda^{-1}.
    \end{cases}
\end{align*}
Set $\delta = \delta'-1 \in [0,1]$ to conclude
\begin{equation*}
    II\lesssim II_1\lesssim \lambda^{1-d-\delta} x^{2-d-\delta} \quad \forall 0\le \delta\le 1.
\end{equation*}
Combining the estimates of $I$ and $II$, the result follows. 
\end{proof}

We may now continue. It is plain that by Hölder's inequality, 
\begin{align*}
    \int_{\Tilde{\mathbb{R}}} k(\lambda|y|)  g(y)d\mu(y) &\le \left(\int_{-\infty}^{-1}+ \int_1^\infty \right) |k(\lambda y)|  |g(y)|d\mu(y)\\
     &\lesssim \|g\|_{p'} \Bigg[\lambda^{p(2-d)}\int_1^{\lambda^{-1}} y^{p(2-d)+d-1}dy  \\
     &+ \lambda^{\frac{p(1-d)}{2}}\int_{\lambda^{-1}}^\infty e^{-p\lambda y}y^{(d-1)(1-p/2)}dy\Bigg]^{1/p} \\
     &\lesssim \lambda^{-d/p} \|g\|_{p'}
\end{align*}
since for $p$ in the range we consider, $p(2-d)+d>0$.

Define operator
\begin{gather*}
    \mathcal{T}: v \mapsto \int_0^1 \nabla u(\cdot,\lambda) \lambda^2F(\lambda) \left[\int_{\Tilde{\mathbb{R}}} k(\lambda|y|)  v(y)d\mu(y) \right] d\lambda.
\end{gather*}
Immediately, we have
\begin{gather}\label{eq_final}
    |\mathcal{T}g(x)| \le C\|g\|_{p'} \int_0^1 |\nabla u(x,\lambda)| \lambda^{2d-2-d/p} d\lambda.  
\end{gather}
To complete the proof, we need to check the boundedness of $\mathcal{T}$ on $L^{p'}$.

$\bullet$ If $d\ge 3$, we choose $0<\delta<1$. By Lemma~\ref{lemma_du},
\begin{equation*}
    \int_0^1 |\nabla u(x,\lambda)| \lambda^{2d-2-d/p} d\lambda \le C x^{2-d-\delta} \int_0^1 \lambda^{d/p'-\delta-1} d\lambda \le C x^{2-d-\delta} 
\end{equation*}
provided $p>d/(d-\delta)$. Consequently
\begin{gather*}
    |\mathcal{J}(f,g)| 
    \le C\|f'\|_p \|g\|_{p'} \left[\int_1^\infty x^{p'(2-d-\delta)+d-1}dx\right]^{1/p'} \le C\|f'\|_p \|g\|_{p'}
\end{gather*}
as long as $p<d/(2-\delta)$. Note that $d'\le d/2 < d/(2-\delta)$ when $d\ge 3$ and therefore, by ranging over all $\|g\|_{p'}= 1$, one deduces
\begin{equation*}
    \|\Delta_d^{1/2}f\|_p \le C\|f'\|_p \quad \frac{d}{d-\delta}<p \le d', \quad 0<\delta<1
\end{equation*}
and the result follows by letting $\delta \to 0$.

$\bullet$ If $2<d<3$, we choose $3-d<\delta<1$ instead. Similarly, by Lemma~\ref{lemma_du}, one ends up with  
\begin{equation*}
    \|\Delta_d^{1/2}f\|_p \le C\|f'\|_p \quad \frac{d}{d-\delta}<p \le d', \quad 3-d<\delta<1,
\end{equation*}
since $d'<d/(2-\delta)$ when $3-d<\delta<1$. We complete the proof by letting $\delta \to 3-d$.

\end{proof}

\begin{remark}\label{remark_6.3.6}
Note that this "harmonic annihilation" method is subjected to the dimension. As demonstrated in the above proof (also see Remark \ref{remark} above), the estimates for $|\nabla u(x,\lambda)|$ become unstable as $d\to 2$. Specifically, when $d=2$, $\lim_{\lambda\to 0}(\Delta_d+\lambda^2)^{-1}(x,y)$ tends to have a logarithmic blow-up at both "ends". This suggests that the 
range of boundedness in Theorem~\ref{thm_d>2} may not be precise for $2<d<3$. We will fix this issue in the next Section~\ref{section_refine}.
\end{remark}

\begin{remark}
One would also be interested about the behaviour of our method to $1<d<2$. Indeed, in this case $\lim_{\lambda \to 0}(\Delta_d+\lambda^2)(x,y)$ converges. However, if we follow the same strategy as in the above proof, the estimate does not give enough decay for $\lambda$ near zero in \eqref{eq_final}. In fact, one may end up with
\begin{equation*}
    |\mathcal{T}g(x)|\lesssim x^{1-d} \|g\|_{p'} \int_0^1 \lambda^{-d/p}d\lambda \quad d<2,
\end{equation*}
which is never converge for $1<p\le d$.
\end{remark}

\section{Refinement: Implicit Harmonic Annihilation}\label{section_refine}

In this section, we improve the result obtained in Theorem~\ref{thm_d>2}. It has been mentioned in Remark~\ref{remark_6.3.6} that the range of \eqref{eq_RRp} for $2<d<3$ may not be sharp. To overcome this problem, let us reconsider the motivation introduced at the beginning of Section~\ref{Section_proof}. Observe that as long as the asymptotic formula \eqref{eq_CCH} holds, the $\nabla$ on the left of \eqref{eq_moti} will always kill the harmonic leading coefficient on the right. So, the questions are 

$(1)$ When to apply "harmonic annihilation"?

$(2)$ How good is the second leading term in \eqref{eq_CCH}? 

$(3)$ How to estimate "$f$" on the left back to $\| |\nabla f| \|_p$?

Therefore, instead of applying "harmonic annihilation" explicitly, we employ it implicitly and then use Hardy inequality. It turns out that we can handle not only the case $2<d<3$, but actually for all $d>1$.

It is well-known that the classical Hardy inequality: for $0<b<\infty$ and $1\le p<\infty$, one has (see for example \cite[Exercise 1.2.8]{G})
\begin{align*}
    \int_0^\infty \left| \int_x^\infty f(t) dt \right|^p x^{b-1} dx \le \left(\frac{p}{b}\right)^p \int_0^\infty |f(t)|^p t^{p+b-1} dt.
\end{align*}
Now, let $f\in C_c^\infty([1,\infty))$, $F(x) = \int_x^\infty f(t) dt$ and $p+b=d$. The above implies
\begin{align*}
    \frac{|F(1)|^p}{b} + \int_1^\infty \left|\frac{F(x)}{x}\right|^p x^{d-1} dx \le \left(\frac{p}{b}\right)^p \int_1^\infty |F'(t)|^p t^{d-1} dt,
\end{align*}
provided $1\le p<d$. Therefore, one infers Hardy inequality on broken line:
\begin{align}\label{Hardy_1D}
    \left\| \frac{f(\cdot)}{|\cdot|} \right\|_{L^p(\Tilde{\mathbb{R}}, d\mu)} \lesssim \|f'\|_{L^p(\Tilde{\mathbb{R}}, d\mu)}, \quad \forall 1\le p< d,
\end{align}
for all $f\in S_0(\Tilde{\mathbb{R}})$.

\begin{theorem}\label{thm_d>2_refine}
Let $d>1$ and $(\Tilde{\mathbb{R}},d\mu)$ be the measure space defined by \eqref{eq_brokenline}. The following reverse Riesz inequality holds
\begin{gather*}
     \|\Delta_d^{1/2}f\|_p \le C\|f'\|_p \quad \forall p\in \begin{cases}
         (1,d)\cup (d,\infty), & 1<d<2,\\
         (1,\infty), & d\ge 2,
     \end{cases}
\end{gather*}
for all $f\in S_0(\Tilde{\mathbb{R}})$.
\end{theorem}

\begin{proof}
Let $d>1$. Recall the notions from Section~\ref{Section_RR_Broken_Line}. By the duality of Theorem~\ref{thm_HS2} and Remark~\ref{remark_all_good}, it suffices to show the following bilinear form (\eqref{eq_bilinear}):
\begin{align*}
    \mathcal{B}(f,g)&:= \left\langle \nabla f, \int_0^1 \nabla [k(\lambda|x|)] \int_{\Tilde{\mathbb{R}}} k(\lambda|y|)g(y)d\mu(y) F(\lambda) d\lambda \right\rangle\\
    &= \int_{\Tilde{\mathbb{R}}} f'(x) \int_0^1 \left[\frac{d}{dx}k(\lambda |x|)\right] \int_{\Tilde{\mathbb{R}}} k(\lambda|y|)g(y)d\mu(y) F(\lambda) d\lambda d\mu(x),
\end{align*}
where $f\in S_0(\Tilde{\mathbb{R}})$ and $g\in \mathcal{C}_c^\infty(\Tilde{\mathbb{R}})$ and for $\lambda\le 1$,
\begin{align*}
    |F(\lambda)|\lesssim \begin{cases}
        \lambda^{d-2}, & 1<d<2,\\
        \frac{-1}{\log{\lambda}}, & d=2,\\
        \lambda^{2d-4}, & d>2,
    \end{cases} 
\end{align*}
is bounded by $C\|f'\|_p \|g\|_{p'}$ for all $1<p<d$. 

Apply integrating by parts and use the fact that $k(\lambda \cdot)$ is a solution to \eqref{eq_DE}, one obtains
\begin{align*}
    \mathcal{B}(f,g)&= \int_{\Tilde{\mathbb{R}}} f(x) \int_0^1 \lambda^2 k(\lambda |x|) \int_{\Tilde{\mathbb{R}}} k(\lambda|y|)g(y)d\mu(y) F(\lambda) d\lambda d\mu(x)\\
    &= \int_{\Tilde{\mathbb{R}}} \frac{f(x)}{|x|} \int_0^1 \lambda^2 F(\lambda) |x| k(\lambda |x|) \int_{\Tilde{\mathbb{R}}} k(\lambda|y|)g(y)d\mu(y) d\lambda d\mu(x).
\end{align*}
Define operator
\begin{align}
    \mathcal{T}: h \mapsto |\cdot|\int_0^1 \lambda^2 F(\lambda) k(\lambda |\cdot|) \int_{\Tilde{\mathbb{R}}} k(\lambda|y|)h(y)d\mu(y) d\lambda, \quad h\in C_c^\infty(\Tilde{\mathbb{R}}).
\end{align}
It follows by a straightforward estimate that for $d>2$ and $1<q<\infty$
\begin{align*}
    |\mathcal{T}g(x)|\lesssim \|g\|_{q'} |x|\int_0^1 \lambda^{2d-2} |k(\lambda |x|)| \left[\int_{\Tilde{\mathbb{R}}} k(\lambda|y|)^q d\mu(y)\right]^{\frac{1}{q}} d\lambda.
\end{align*}
Use asymptotic formula (below \eqref{eq_DE}). It is clear that
\begin{align}\nonumber
    |\mathcal{T}g(x)|&\lesssim \|g\|_{q'} |x| \int_0^1 |k(\lambda|x|)| \lambda^{2d-2-\frac{d}{q}} d\lambda\\ \nonumber
    &\lesssim  \|g\|_{q'} \left[ |x|^{3-d} \int_0^{|x|^{-1}} \lambda^{\frac{d}{q'}} d\lambda + |x|^{\frac{3-d}{2}} \int_{|x|^{-1}}^{\infty} e^{-\lambda|x|} \lambda^{\frac{3(d-1)}{2}-\frac{d}{q}} d\lambda \right] \\
    &\lesssim |x|^{2-d-\frac{d}{q'}} \|g\|_{q'}.
\end{align}
One checks easily for $d>2$
\begin{align*}
    \|\mathcal{T}g\|_{q'} \lesssim \|g\|_{q'} \left[\int_1^\infty r^{q'\left(2-d-\frac{d}{q'}\right)} r^{d-1}dr\right]^{\frac{1}{q'}} \lesssim \|g\|_{q'},
\end{align*}
as long as $q'(2-d)<0$. Particularly, $\mathcal{T}$ is bounded on $L^{p'}$ for all $1<p<d$.

While for $1<d\le 2$, a similar estimate (using different asymptotic formula) only gives that for all $|x|\ge 1$
\begin{align*}
    |\mathcal{T}g(x)|&\lesssim \|g\|_{q'} |x|^{-d/q'}, \quad 1<q<\infty.
\end{align*}
However, it is enough to verify our result. Indeed, this implies for any $\delta>0$,
\begin{align*}
    \mu \left(\left\{x\in \Tilde{\mathbb{R}}; |\mathcal{T}g(x)|>\delta \right\}\right) &\le \mu \left(\left\{x\in \Tilde{\mathbb{R}}; C |x|^{-d/q'} \|g\|_{q'}>\delta \right\}\right)\\
    &= \mu \left(\left\{x\in \Tilde{\mathbb{R}};  |x| < C \left(\frac{\|g\|_{q'}}{\delta}\right)^{q'/d} \right\}\right)\\
    &\lesssim \delta^{-q'} \|g\|_{q'}^{q'},
\end{align*}
i.e. $\mathcal{T}$ is of weak type $(q',q')$ for all $1<q<\infty$. It follows by interpolation, one concludes that for all $d>1$, in particular, $\mathcal{T}$ is bounded on $L^{p'}$ for all $1<p<d$.

Now, we may end the proof by Hardy inequality \eqref{Hardy_1D}, for $1<p<d$,
\begin{align*}
    |\mathcal{B}(f,g)|\le \left\| \frac{f(\cdot)}{|\cdot|} \right\|_{p} \|\mathcal{T}g\|_{p'}\lesssim \|f'\|_p \|g\|_{p'}
\end{align*}
as desired.
\end{proof}


\section{Reverse Riesz in the Critical Dimension}\label{section_6}

In this section, we use techniques built in previous sections to prove \eqref{eq_RRp} on the critical case studied in \cite{HNS}. That is we consider manifold with ends: 
\begin{align*}
    \mathcal{M} = (\mathbb{R}^2\times \mathcal{M}_-)\#(\mathbb{R}^{n_+}\times \mathcal{M}_+),
\end{align*}
where $n_+\ge 3$ and $2+\textrm{dim}(\mathcal{M}_-) = n_+ +\textrm{dim}(\mathcal{M}_+) = N$ and $\mathcal{M}_{\pm}$ are two smooth compact manifolds. In \cite{HNS}, the authors showed that the Riesz transform $\nabla \Delta^{-1/2}$ on such manifold is bounded on $L^p$ if and only if $1<p\le 2$ and is of weak type $(1,1)$. With no ambiguity, we use notions: $E_{\pm}:= \mathbb{R}^{n_{\pm}}\times \mathcal{M}_{\pm} \setminus K_{\pm}$ (here $n_-=2$), where $K_{\pm} \subset \mathbb{R}^{n_{\pm}}\times \mathcal{M}_{\pm}$ are some compact sets. 

Recall formula
\begin{align*}
    \nabla \Delta^{-1/2} = \frac{2}{\pi} \int_0^{k_0} \nabla (\Delta+k^2)^{-1}dk + \frac{2}{\pi} \int_{k_0}^\infty \nabla (\Delta+k^2)^{-1}dk := R_L + R_H,
\end{align*}
where $k_0\in (0,1/2)$ is some fixed number determined by parametrix and we omit $2/\pi$ from now on. In the same spirit, one uses resolution to identity
\begin{align*}
    \Delta^{1/2} = \int_0^\infty \Delta (\Delta+k^2)^{-1}dk.
\end{align*}

To maintain uniformity in the notation of \cite{HNS}, we introduce following notions. Let $|z|:= \sup_{a\in K}d(z,a)$. Apparently, $|z|\gtrsim 1$, which is away from zero. Define function
\begin{align*}
    \textrm{ilg}k := \begin{cases}
        \frac{-1}{\log{k}}, & 0<k\le 1/2,\\
        0, & k=0.
    \end{cases}
\end{align*}
Recall the following key lemma in \cite{HNS}. Compare it to Lemma~\ref{leKey}.

\begin{lemma}\label{lemma_key_critical_dimension}\cite[Lemma 2.14]{HNS}
Let $v \in \mathcal{C}_c^{\infty}(\mathcal{M})$. Suppose $\psi$ is the solution to the equation $\Delta \psi = -v$. Then, for every integer $q$, there  exists an approximate solution $u(z,k)$ such that $(\Delta + k^2)u = v$ in the sense that
\begin{align*}
    (\Delta + k^2)u - v = O((\textrm{ilg}k)^q |z|^{-\infty}).
\end{align*}
Moreover, one has estimates
\begin{gather}\label{key_critical_dimension}
|u(z, k)| \le \begin{cases}
    C, & z\in K,\\
    C |z|^{2-n_+} e^{- ck |z|}, & z \in  E_+,\\
    C e^{-ck |z|}, & z \in E_-.
\end{cases}\\  \label{key_critical_dimension2}
|\nabla u(z, k)| \le \begin{cases}
    C, & z\in K,\\
    C |z|^{1-n_+} e^{- ck |z|}, & z \in  E_+,\\
    C \left[|z|^{-2} + (\textrm{ilg}k)|z|^{-1}\right] e^{-ck|z|}, & z \in E_-
\end{cases}
\end{gather}
for all $0\le k\le 1/2$.
\end{lemma}

With the help of this lemma, one can again decompose the resolvent operator 
\begin{align*}
    (\Delta+k^2)^{-1} = \sum_{i=1}^4 G_i(k),\quad \forall 0<k\le k_0.
\end{align*}
We have similar properties as in the case $n_*\ge 3$. For detailed discussion about parametrix, we refer readers to \cite[Section~2,3]{HNS}.

$\bullet$ For $i=2,4$, and $h\in C_c^\infty(\mathcal{M})$,
\begin{align}\label{eq_critical_G24}
    \left\| \int_0^{k_0} \nabla G_i(k) h dk \right\|_q \lesssim \|h\|_q,\quad \forall 1<q<\infty.
\end{align}

$\bullet$ For $i=1,3$, denote by $\Delta_{+}, \Delta_-$ the Laplace-Beltrami operator on $\mathbb{R}^{n_+}\times \mathcal{M}_+$ and $\mathbb{R}^2\times \mathcal{M}_-$ respectively. Define $\phi_{\pm}\in \mathcal{C}^\infty(\mathcal{M})$ such that
\begin{align*}
    &(1) \textrm{supp}(\phi_{\pm}) \subset E_{\pm}, \quad (2) 0\le \phi_{\pm} \le 1,\\
    &(3) \phi_{\pm} = 1\quad \textrm{outside some compact set $\Tilde{K}_{\pm}$},
\end{align*}
where $K_{\pm}\subset \subset \Tilde{K}_{\pm}$. Then, one can write down explicit formulas ($z_{\pm}^0 \in K_{\pm}$ are two fixed points playing the role of origin):
\begin{align}\label{eq_critical_G1}
G_1(k)(z,z') &= \sum_{\pm} (\Delta_{\pm} +k^2)^{-1}(z,z') \phi_{\pm}(z)\phi_{\pm}(z'),\\ \label{eq_critical_G3}
G_3(k)(z,z') &= \sum_{\pm} (\Delta_{\pm} +k^2)^{-1}(z_{\pm}^0,z') u_{\pm}(z,k)\phi_{\pm}(z'),
\end{align}
where $u_{\pm}$ are defined by Lemma~\ref{lemma_key_critical_dimension} such that $(\Delta+k^2)u_{\pm} = v_{\pm}$ and $v_{\pm} = -\Delta \phi_{\pm}\in \mathcal{C}_c^\infty(\mathcal{M})$.

$\bullet$ One has estimates: for $k>0$
\begin{align}\label{eq_critical_resol}
    &(\Delta_- + k^2)^{-1}(z,z') \lesssim \left[d(z,z')^{2-N} + 1 + |\log{kd(z,z')}|\right] \textrm{exp}\left(-ck d(z,z')\right),\\
    &\left|\nabla (\Delta_-+k^2)^{-1}(z,z')\right| \lesssim \left[d(z,z')^{1-N} + d(z,z')^{-1}\right] \textrm{exp}\left(-ck d(z,z')\right),\\ \label{eq_critical_resol3}
    &(\Delta_+ + k^2)^{-1}(z,z') \lesssim \left[d(z,z')^{2-N} + d(z,z')^{2-n_+}\right] \textrm{exp}\left(-ck d(z,z')\right),\\ \label{eq_critical_resol2}
    &\left|\nabla (\Delta_+ +k^2)^{-1}(z,z')\right| \lesssim \left[d(z,z')^{1-N} + d(z,z')^{1-n_+}\right] \textrm{exp}\left(-ck d(z,z')\right).
\end{align}

We need following lemma. Compare to \cite[Proposition~2.1]{DR}. We mention that in the rest of the paper, we may write $dz = d\mu(z)$ for simplicity.

\begin{lemma}\label{thm_Hardy_M}
Let $\mathcal{M}$ be a manifold with ends defined by \eqref{eq_Manifold} but with $l\ge 2$ and $n_*\ge 2$. Then, the following Hardy inequality holds
\begin{align*}
    \int_{\mathcal{M}} \left|\frac{f(z)}{|z|}\right|^p dz \lesssim \int_{\mathcal{M}} |\nabla f|^p dz,\quad \forall 1< p<n_*
\end{align*}
for all $f\in \mathcal{C}_c^\infty(\mathcal{M})$.
\end{lemma}

\begin{proof}[Proof of Lemma~\ref{thm_Hardy_M}]
We follow the idea from \cite[Proposition~2.1, Theorem~2.3]{DR}. Let $\{\varphi_i\}_{1\le i\le l}$ be a sequence of smooth functions such that $\textrm{supp}(\varphi_i)\subset E_i$ and $\varphi_i=1$ outside some compact set  containing $K_i\subset \mathbb{R}^{n_i}\times \mathcal{M}_i$. Decompose 
\begin{equation*}
    f = \sum_{1\le i\le l} f\varphi_i + f\left(1-\sum_{1\le i\le l}\varphi_i \right) := \sum_{1\le i\le l}f_i + f_{0}.
\end{equation*}
It is clear to see that for each $1\le i\le l$, $f_i\in \mathcal{C}_c^\infty(\mathbb{R}^{n_i}\times \mathcal{M}_i)$ (by zero extension) and $f_0\in \mathcal{C}_c^\infty(\Tilde{K})$, where $\Tilde{K}$ is a small neighbourhood of the compact connection set $K$. It is easy to check that for each $1\le i\le l$, all the assumptions in \cite[Theorem~2.3]{DR} hold on $\mathbb{R}^{n_i}\times \mathcal{M}_i$ (for the case $n_*\ge 3$, see \cite{HS} and see \cite{HNS} for $n_*=2$). Particularly, for each $i$, one has reverse doubling property with exponent $n_i$, i.e. for all $z\in \mathbb{R}^{n_i}\times \mathcal{M}_i$ and $R\ge r>0$,
\begin{align*}
    \frac{V(z,R)}{V(z,r)} \gtrsim \left(\frac{R}{r}\right)^{n_i},
\end{align*}
where $V(z,r) = \mu(B(z,r))$ denotes the volume of the ball. Therefore, \cite[Theorem~2.3]{DR} tells us (recall that $|z|\gtrsim 1$ is away from zero)
\begin{align}\label{eq_hardy}
    \int_{\mathbb{R}^{n_i}\times \mathcal{M}_i} \left| \frac{f_i(z)}{|z|} \right|^p dz \lesssim \int_{\mathbb{R}^{n_i}\times \mathcal{M}_i} |\nabla f_i(z)|^p dz \lesssim \||\nabla f|\|_p^p  + \int_{\Tilde{K}}|f(z)|^p dz,\quad \forall 1< p<n_i.
\end{align}
Combine this with the trivial estimate
\begin{align*}
    \int_{\mathcal{M}} \left| \frac{f_0(z)}{|z|} \right|^p dz \lesssim \int_{\Tilde{K}} |f(z)|^p dz,\quad \forall 1< p<n_*.
\end{align*}
One only needs to establish
\begin{align*}
    \|f\|_{L^p(\Tilde{K})} \le C_{\Tilde{K}} \||\nabla f|\|_p,\quad 1< p<n_*.
\end{align*}
Note that for $1<p<n_*$, this is a direct consequence of $p\textit{-}$hyperbolicity, see for example \cite{Marc}. Observe that on each end, the inequality \eqref{eq_hardy} itself implies $\mathbb{R}^{n_i}\times \mathcal{M}_i$ is $p\textit{-}$hyperbolic for all $1<p<n_i$. Hence, it follows by \cite[Corollary~2.1]{Devyver_perturbation} that $\mathcal{M}$ is $p\textit{-}$hyperbolic for all $1<p<\max_i n_i$. The proof is then complete.
\end{proof}

Next, we prove Theorem~\ref{thm_RR_critical}. Let us recall the statement here.

$\bullet$ \textbf{Statement of Theorem~\ref{thm_RR_critical}:} Let $\mathcal{M}$ be a manifold with ends defined by \eqref{eq_manifold_criticalcase}. Then \eqref{eq_RRp} holds for all $1<p<\infty$.

\begin{proof}[Proof of Theorem~\ref{thm_RR_critical}]
The first part of the proof is similar to Section~\ref{Section_proof}. By \cite[Theorem~1.2]{HNS} and duality, one only needs to confirm \eqref{eq_RRp} for $1<p<2$. Let $f,g\in \mathcal{C}_c^\infty(\mathcal{M})$. By the self-adjointness and non-negativity of $\Delta$, and the $L^p$ boundedness of $R_H$ (\cite[Proposition~4.5]{HNS}) and \eqref{eq_critical_G24}, one infers that
\begin{align*}
    \langle \Delta^{1/2}f, g\rangle &= \left\langle \nabla f, \int_0^\infty \nabla (\Delta+k^2)^{-1}g dk \right\rangle\\
    &= \langle \nabla f, R_H g\rangle + \left\langle \nabla f,  \sum_{i=1}^4 \int_0^{k_0} \nabla G_i(k)g dk \right\rangle\\
    &\lesssim \||\nabla f|\|_p \|g\|_{p'} + \left| \left\langle \nabla f,  \sum_{i=1,3} \int_0^{k_0} \nabla G_i(k)g dk \right\rangle \right|.
\end{align*}
Therefore, it suffices to treat the second term. For the term involving $G_1$, it has been showed in \cite[Proposition~4.4]{HNS} that part of this term, i.e. operator with kernel (use formulas \eqref{eq_critical_G1} and \eqref{eq_critical_resol}-\eqref{eq_critical_resol2})
\begin{align*}
    \phi_{\pm}(z) \int_0^{k_0} \nabla (\Delta_{\pm}+k^2)^{-1}(z,z') dk \phi_{\pm}(z'),
\end{align*}
acts as a bounded operator on $L^q$ for all $1<q<\infty$.

Therefore, by the explicit formulas \eqref{eq_critical_G1}, \eqref{eq_critical_G3}, we are left to consider the following bilinear form
\begin{align*}
    \mathcal{I}(f,g):= \sum_{\pm} \int_{\mathcal{M}} \nabla f(z) \cdot &\Bigg(\int_{\mathcal{M}} \int_0^{k_0} \nabla \phi_{\pm}(z) (\Delta_{\pm}+k^2)^{-1}(z,z')\phi_{\pm}(z') \\
    &+ \nabla u_{\pm}(z,k) (\Delta_{\pm}+k^2)^{-1}(z_{\pm}^0,z')\phi_{\pm}(z') g(z') dk dz'\Bigg) dz,
\end{align*}
and we wish to establish $|\mathcal{I}(f,g)|\le C \|\nabla f\|_p \|g\|_{p'}$ for $1<p<2$.

Similar to Section~\ref{Section_proof}, one can split
\begin{align*}
    (\Delta_{\pm}+k^2)^{-1}(z,z') = \left[(\Delta_{\pm}+k^2)^{-1}(z,z') - (\Delta_{\pm}+k^2)^{-1}(z_{\pm}^0,z')\right] + (\Delta_{\pm}+k^2)^{-1}(z_{\pm}^0,z').
\end{align*}
It follows by a similar argument in Lemma~\ref{lemma_difference_Lq}, also see \cite{HNS}, that the difference term acts as a bounded operator on $L^q$ $(1<q<\infty)$, i.e.
\begin{align*}
    \left\| h \mapsto \int_{\mathcal{M}} \nabla \phi_{\pm}(\cdot) \int_0^{k_0} \left[(\Delta_{\pm}+k^2)^{-1}(z,z') - (\Delta_{\pm}+k^2)^{-1}(z_{\pm}^0,z')\right] dk \phi_{\pm}(z') h(z') dz'   \right\|_q \lesssim \|h\|_q
\end{align*}
for all $h\in \mathcal{C}_c^\infty(\mathcal{M})$.

Consequently, one only needs to treat 
\begin{align*}
    \mathcal{J}_{\pm}(f,g):= \int_{\mathcal{M}} \nabla f(z) \cdot \int_{\mathcal{M}} \int_0^{k_0} \left[\nabla \phi_{\pm}(z) + \nabla u_{\pm}(z,k) \right] (\Delta_{\pm}+k^2)^{-1}(z_{\pm}^0,z')\phi_{\pm}(z')g(z') dk dz' dz.
\end{align*}

Apply integration by parts and use the fact that $\Delta u_{\pm} = v_{\pm} - k^2 u_{\pm} = -\Delta \phi_{\pm} - k^2 u_{\pm}$. One obtains
\begin{align*}
    \mathcal{J}_{\pm}(f,g) &= \int_{\mathcal{M}} f(z) \int_{\mathcal{M}} \int_0^{k_0} \left[-k^2 u_{\pm}(z,k) \right] (\Delta_{\pm}+k^2)^{-1}(z_{\pm}^0,z')\phi_{\pm}(z')g(z') dk dz' dz\\
    &= \int_{\mathcal{M}} \frac{f(z)}{|z|} \int_{\mathcal{M}} \int_0^{k_0} |z| \left[-k^2 u_{\pm}(z,k) \right] (\Delta_{\pm}+k^2)^{-1}(z_{\pm}^0,z')\phi_{\pm}(z')g(z') dk dz' dz.
\end{align*}
Define operator
\begin{align*}
    \mathcal{T}_{\pm}: h \mapsto |\cdot| \int_{\mathcal{M}} \int_0^{k_0} k^2 u_{\pm}(\cdot,k) (\Delta_{\pm}+k^2)^{-1}(z_{\pm}^0,z')\phi_{\pm}(z')h(z') dk dz'.
\end{align*}

\begin{lemma}\label{lemma_critical}
Under the assumptions of Theorem~\ref{thm_RR_critical}, $\mathcal{T}_{\pm}$ are bounded on $L^{q'}$ for all $1<q<\infty$.
\end{lemma}

\begin{proof}[Proof of Lemma~\ref{lemma_critical}]
By a straightforward estimate \eqref{eq_critical_resol},
\begin{align*}
    |\mathcal{T}_{-}h(z)| &\lesssim |z| \int_0^{k_0} k^2 |u_{-}(z,k)| \int_{E_{-}} \left(1 + |\log{(k |z'|)}|\right) e^{-ck|z'|} |h(z')| dz'dk\\
    &\lesssim |z| \|h\|_{q'} \int_0^{k_0} k^2 |u_{-}(z,k)| \left(\int_{E_{-}} \left(1 + |\log{(k |z'|)}|\right)^q e^{-c' k |z'|}dz' \right)^{1/q}dk\\
    &\lesssim |z| \|h\|_{q'} \int_0^{k_0} k^2 |u_{-}(z,k)| \left(\int_1^\infty \left(1 + |\log{(k r)}|\right)^q e^{-c' k r} r dr \right)^{1/q}dk\\
    &\lesssim |z| \|h\|_{q'} \int_0^{k_0} k^{2/q'} |u_{-}(z,k)| dk.
\end{align*}
By Lemma~\ref{lemma_key_critical_dimension}, one has estimates
\begin{align}\label{eq_u_+}
    |u_{\pm}(z,k)| \lesssim \begin{cases}
    1, & z\in K,\\
    |z|^{2-n_+} e^{- ck |z|}, & z \in  E_+,\\
    e^{-ck |z|}, & z \in E_-.
\end{cases}
\end{align}
Therefore,
\begin{align*}
    |\mathcal{T}_{-}h(z)| \lesssim \|h\|_{q'} \begin{cases}
        1, & z\in K,\\
        |z|^{2-n_+-2/q'}, & z\in E_+,\\
        |z|^{-2/q'}, & z\in E_-.
    \end{cases}
\end{align*}
Decompose $\mathcal{T}_- = \chi_K \mathcal{T}_-  + \chi_{E_+} \mathcal{T}_-  +  \chi_{E_-} \mathcal{T}_- :=  \mathcal{T}_{-}^1 + \mathcal{T}_{-}^2 + \mathcal{T}_{-}^3$. It is clear that by Hölder's inequality that $\|\mathcal{T}_{-}^1 + \mathcal{T}_{-}^2\|_{q'\to q'}\lesssim 1$ for all $1<q<\infty$ since $n_+ > 2$. While for $\mathcal{T}_{-}^3$, note that for $\lambda >0$,
\begin{align*}
    \mu \left(\left\{z\in E_-; |\mathcal{T}_-^3 h(z)|>\lambda \right\}\right) &\le \mu \left(\left\{z\in E_-; C |z|^{-2/q'} \|h\|_{q'}>\lambda \right\}\right)\\
    &= \mu \left(\left\{z\in E_-;  |z| < C \left(\frac{\|h\|_{q'}}{\lambda}\right)^{q'/2} \right\}\right)\\
    &\lesssim \lambda^{-q'} \|h\|_{q'}^{q'},
\end{align*}
i.e. $\mathcal{T}_-^3$ is of weak type $(q', q')$ for all $1<q<\infty$. It follows by interpolation, one concludes $\mathcal{T}_-$ is bounded on $L^{q'}$ for all $1<q<\infty$.

The case $\mathcal{T}_+$ is similar. Indeed, by using \eqref{eq_critical_resol3} instead , one comes up with 
\begin{align*}
    |\mathcal{T}_+h(z)|\lesssim |z| \|h\|_{q'} \begin{cases}
        \int_0^{k_0} |u_+(z,k)| k^{n_+/q'} dk, & 1<q<\frac{n_+}{n_+ -2},\\
        \int_0^{k_0} |u_+(z,k)| k^{1/q'+1} dk, & q\ge \frac{n_+}{n_+ -2}.
    \end{cases}
\end{align*}
In the virtue of \eqref{eq_u_+}, one deduces
\begin{align*}
    |\mathcal{T}_+h(z)|&\lesssim \|h\|_{q'} \Bigg(\chi_K(z) + \chi_{E_+}(z) \begin{cases}
        |z|^{2-n_+-n_+/q'}, & 1<q<\frac{n_+}{n_+-2},\\
        |z|^{1 -n_+ - 1/q'}, & q\ge \frac{n_+}{n_+-2}.
    \end{cases}\\
    &+ \chi_{E_-}(z) \begin{cases}
        |z|^{-n_+/q'}, & 1<q<\frac{n_+}{n_+-2},\\
        |z|^{-1/q' -1}, & q\ge \frac{n_+}{n_+-2}.\\
    \end{cases}\Bigg).
\end{align*}
One checks easily that $\|\mathcal{T}_+h\|_{q'} \lesssim \|h\|_{q'}$ for all $1<q<\infty$ and the proof is complete.
  
\end{proof}

Now, we may continue by Lemma~\ref{lemma_critical} that 
\begin{align*}
    |\mathcal{J}_{\pm}(f,g)|&\le \int_{\mathcal{M}} \frac{|f(z)|}{|z|} |\mathcal{T}_{\pm}g(z)| dz \\
    &\le \left\| \frac{f(\cdot)}{|\cdot|}\right\|_p \|\mathcal{T}_{\pm}g\|_{p'}\\
    &\lesssim \||\nabla f|\|_p \|g\|_{p'},\quad \forall 1<p<2,
\end{align*}
where the last inequality follows by $L^p\textit{-}$Hardy inequality, Lemma~\ref{thm_Hardy_M}. The proof of Theorem~\ref{thm_RR_critical} follows by ranging all $g\in \mathcal{C}_c^\infty(\mathcal{M})$ with $\|g\|_{p'}=1$.

\end{proof}

{\bf Acknowledgments.} 
This note is part of the author's PhD thesis. I would like to thank my supervisor Adam Sikora for introducing me to the topic, carefully reading the note and giving valuable advice. I also want to thank Professor Andrew Hassell for suggestions on the references and encouragement.


\bibliographystyle{abbrv}

\bibliography{references.bib}

\end{document}